\documentclass[11pt]{amsart}
\usepackage[margin=3cm]{geometry}
\usepackage{bm}
\usepackage{amsthm}
\usepackage{amssymb}
\usepackage{mathrsfs}
\numberwithin{equation}{section}
\usepackage{tikz}
\usepackage{tkz-euclide}
\usepackage{graphicx}
\usepackage[colorlinks=true, pagebackref]{hyperref}
\usepackage{xcolor}

\newtheorem{theorem}{Theorem}[section]
\newtheorem{definition}[theorem]{Definition}
\newtheorem{conjecture}[theorem]{Conjecture}

\newtheorem{remark}[theorem]{Remark}

\newcommand{\R}{\mathbb{R}}

\newcommand{\Z}{\mathbb{Z}}
\newcommand{\e}{\epsilon}

\newcommand{\pr}{\partial}
\newcommand{\tf}{\tilde{f}}

\DeclareMathOperator{\area}{Area}

\newtheorem{thm}[theorem]{Theorem}
\newtheorem*{thm*}{Theorem}
\newtheorem{lem}[theorem]{Lemma}
\newtheorem{prop}[theorem]{Proposition}
\newtheorem{coro}[theorem]{Corollary}

\theoremstyle{definition}

\newtheorem{rmk}[theorem]{Remark}

\begin{document}

\begin{abstract}
We construct a new example of an immortal mean curvature flow of smooth embedded connected surfaces in $\R^3$, which converges to a plane with multiplicity $2$ as time approaches infinity.

    
\end{abstract}

\title{Mean curvature flow with multiplicity $2$ convergence in $\mathbb{R}^3$}

\author{Jingwen Chen, Ao Sun}
\address{Department of Mathematics, University of Pennsylvania,
David Rittenhouse Lab,
209 South 33rd Street,
Philadelphia, PA 19104}
\email{jingwch@sas.upenn.edu}
\address{Department of Mathematics, Lehigh University, Chandler-Ullmann Hall, Bethlehem, PA 18015}
\email{aos223@lehigh.edu}
\date{}
\maketitle

\section{Introduction}

Higher multiplicity convergence is a significant phenomenon in differential geometry and geometric measure theory. Roughly speaking, it means that several different sheets of a surface collapse to the same sheet in the limit. One example is the dilation of the minimal surface catenoid, which ultimately converges to a plane with multiplicity $2$. 

When higher multiplicity convergence occurs, the geometry and topology may change dramatically. In the case of the dilating catenoid converging to a plane with multiplicity $2$, the nontrivial topology and high-curvature region of the catenoids vanish in the limit plane. Therefore, higher multiplicity convergence is a central topic in the study of geometry and topology using geometric measure theory.

In this paper, we construct a new example of the mean curvature flow of smooth embedded connected surfaces in $\R^3$, which converges to a plane with multiplicity $2$ as time approaches $\infty$. A smooth one-parameter family of hypersurfaces $M(t) \subset \R^{n+1},\ t \in (0, T)$ evolves by its mean curvature if 
\[
\pr_t x=\vec H(x),
\]
where $\vec H(x)$ is the mean curvature vector of the surface at point $x$. A smooth mean curvature flow that exists for all future time is called {\bf immortal}.

\begin{thm}\label{thm:main}
There exists an immortal mean curvature flow of connected embedded rotationally symmetric surfaces $(M(t))_{t\in[0,\infty)}$ in $\R^3$, which is smooth for $t > 0$, whose limit as $t\to\infty$ is a plane with multiplicity $2$.
\end{thm}

Our example contrasts sharply with the picture of the singular behavior of mean curvature flow in $\R^3$. Recall that Huisken \cite{Huisken90_asymptotic} introduced the rescaled mean curvature flow to blow up the singularities of mean curvature flow. If $(y,T)$ is a singularity of a mean curvature flow $\{M(t)\}_{t\in[0,T)}$, the rescaled mean curvature flow is defined to be $\widetilde{M}(\tau):=e^{\tau/2}(M({T-e^{-\tau}})-y)$, and it satisfies the equation 
\[
\pr_\tau x=\vec H(x)+{x^\perp}/2,
\]
where $x^\perp$ denotes the projection of the position vector to the normal direction. Using Huisken's monotonicity formula, Ilmanen \cite{Ilmanen95_Sing2D} and White \cite{White97_Stratif} showed that the (subsequential) limit of a rescaled mean curvature flow is a self-shrinker, namely a (possibly singular) hypersurface satisfying the equation $\vec H(x)+{x^\perp}/2=0$. It is worth noting that the plane is one of the self-shrinkers. The limit self-shrinker of a rescaled mean curvature flow may have higher multiplicity. In \cite{Ilmanen95_Sing2D}, Ilmanen made the following conjecture:
\begin{conjecture}[Multiplicity One Conjecture of mean curvature flow]
For a mean curvature flow of closed embedded surfaces in $\R^3$, the rescaled mean curvature flow blowing up a singularity must converge to a self-shrinker with multiplicity $1$.
\end{conjecture}

This conjecture has been verified in many special cases, for example when the initial surface is mean convex or the blow-up rate of the mean curvature is at most type I. We refer the readers to \cite{HuiskenSinestrari99_CVPDE, HuiskenSinestrari99_convexity, White00_size, White03_nature, HaslhoferKleiner17_mean, LiWang19_extension-MCF, LiWang22_Multi-One-MCF}. 

While we were preparing this paper, Bamler-Kleiner \cite{BamlerKleiner23_multiplicity} posted a proof confirming the Multiplicity One Conjecture. Consequently, a smooth rescaled mean curvature flow can not converge to the plane with higher multiplicity as time approaches $\infty$, which contrasts sharply with our example. One reason for this lies in the stability of the plane as a static point in mean curvature flow, meanwhile, the plane is unstable as a static point in the rescaled mean curvature flow, as observed by Colding-Minicozzi in \cite{ColdingMinicozzi12_GenericMCF}.

To elaborate, imagine over a large region of the plane, the flow converging to it with multiplicity $2$ can be represented as double graphs. Let the difference between these two graphs be denoted by $f$. This function satisfies a nonlinear equation, with its linear part given by $\pr_t f=Lf$, where $L$ is the linearized operator of the plane. For mean curvature flow, $L=\Delta$ is the Laplacian, and for rescaled mean curvature flow, $L=\Delta-1/2\langle x,\nabla\cdot\rangle+1/2$. Because $f>0$, its evolution is dominated by the first eigenfunction of $L$, and the first eigenvalue determines its growth rate. For $L=\Delta$, the first eigenvalue is $0$, allowing the double graphs to tend to each other due to the nonlinear factor, with a sub-exponential speed. However, for $L=\Delta-1/2\langle x,\nabla\cdot\rangle+1/2$, the first eigenvalue is $1$, implying that the double graphs will move away from each other with an exponential speed. 

We would like to compare our example with another higher multiplicity example in the min-max theory, a significant tool for constructing minimal hypersurfaces in Riemannian manifolds. Recent study by Marques-Neves \cite{MarquesNeves16_Index, MarquesNeves17_Infinite, MarquesNeves18_Index} and Xin Zhou \cite{Zhou19_multi1} provided a comprehensive description of the min-max minimal hypersurfaces in Riemannian manifolds with dimensions between $3$ and $7$. These minimal hypersurfaces are either unstable with multiplicity $1$, or stable with potentially higher multiplicity. 

In \cite{WangZhou22_minmax-higher-multiplicity}, Zhichao Wang and Xin Zhou constructed a closed Riemannian manifold to show the existence of higher multiplicity min-max minimal surfaces. While the min-max theory involves a different mechanism, the reason that higher multiplicity is a feature of stable minimal surfaces is similar: when a minimal surface is unstable, the linearized operator has a positive eigenfunction with negative eigenvalue, which pushes the nearby graphs away from each other, preventing higher multiplicity convergence.

Recall that, after taking a quotient of the rotation, the Euclidean space becomes a half-plane. A rotationally symmetric surface can then be represented by a curve in the half-plane, known as the profile curve. When considering the evolution of the profile curve, our example can be interpreted as a free boundary curve shortening flow in a surface with a boundary, up to a conformal factor in the speed. In our case, the long-time behavior of our flow is ``collapsing'', which is very different from the free boundary curve shortening flow in a surface with convex boundary as studied in \cite{LangfordZhu23_distance, ko2023existence}.

We remark that there was an example by Evans and Ilmanen in \cite[Appendix E]{Ilmanen94_elliptic}, showing that a Brakke flow can have a higher density limit as $t\to\infty$ even if initially the Brakke flow is supported on a smooth embedded surface with density $1$. A Brakke flow is a geometric measure theoretic weak flow of mean curvature flow, and the density is the ratio of the measure compared to the Hausdorff measure. To our knowledge, the example in \cite[Appendix E]{Ilmanen94_elliptic} is not found in the existing literature, and Ilmanen did not explain the method to construct such an example. While our example is inspired by their description, it should be noted that from the perspective of Brakke flows, the example suggested by Evans and Ilmanen might not be smoothly connected.

\subsection{Shape of the immortal flow}

The flow constructed in Theorem \ref{thm:main} looks like two parallel planes connected by a neck. As time goes by, the size of the neck shrinks and shrinks, and as time to infinity, the size of the neck tends to $0$ (See the Figure \ref{fig:finalflow} below).

Although this may seem intuitive, showing that such a flow converges to a plane with multiplicity $2$ as time tends to infinity is not straightforward. One challenge is to prove that the neck does not pinch at a finite time, and the flow does not expand to infinity. One novelty of this paper is to turn intuition into mathematically rigorous arguments.

\begin{figure}[htbp]
    \centering
    \includegraphics[width=0.7\textwidth]{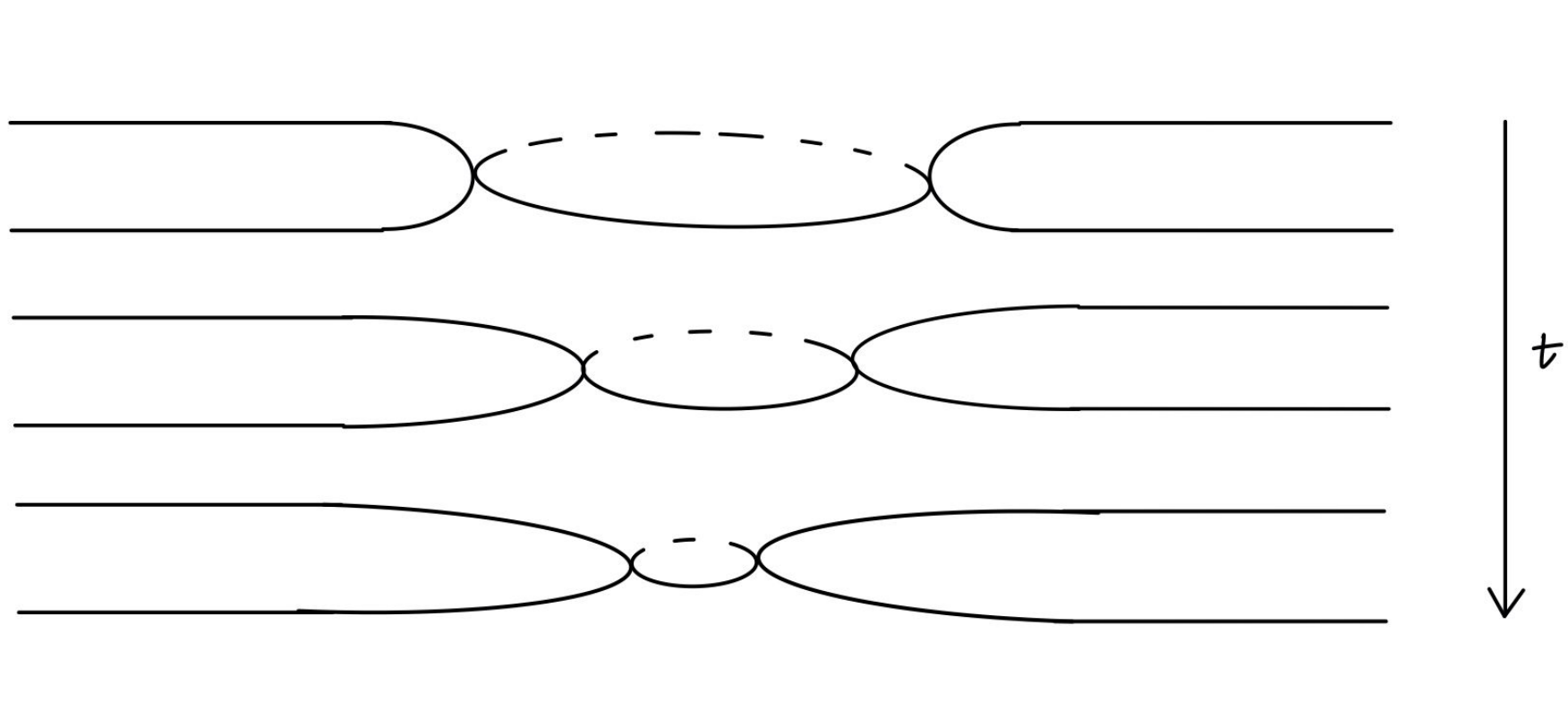}
    \caption{The shape of the immortal flow as time goes by. The lower figure shows the flow at a larger time.}
    \label{fig:finalflow}
\end{figure}

There are several tricky facts about this immortal flow. Firstly, our flow appears as two sheets, with each sheet asymptotically approaching a plane. As detailed in the Appendix, the asymptotically planar surface will be asymptotic to the same plane under mean curvature flow. Consequently, the immortal flow does not uniformly converge to the plane with multiplicity $2$, but instead converges uniformly in any compact region.

Secondly, this flow seems highly unstable. In fact, our construction reveals that even a minor perturbation can lead to a finite-time neck pinch singularity or cause the flow to escape to infinity. Therefore, we anticipate that such higher multiplicity convergence is non-generic.

Thirdly, while the immortal flow from our construction appears to be bounded between two parallel planes, we expect that a slight modification could yield an immortal flow converging to a plane with multiplicity $2$ whose ends grow, rather than being bounded. It is not clear how much the ends can grow. Our method suggests that if the ends grow too fast, the immortal flow might simply become a static catenoid (see Remark \ref{rmk:different family}).

\subsection{Idea of the construction}
We will be focused on rotationally symmetric mean curvature flow, which is a class of mean curvature flow that has been well-studied by Altschuler, Angenent, and Giga in \cite{AAG}. Our construction involves two main steps.

In the first step, we construct a smooth free boundary mean curvature flow within a solid cylinder with radius $R$ in $\R^3$. This flow converges to a free boundary disk with multiplicity $2$ as time tends to infinity. Given a domain $D$ with boundary $\pr D$, we say that a family of hypersurfaces with boundary $\{M(t)\}_{t\in[0,T)}$ evolves by its mean curvature with free boundary condition, if $\pr M_t\subset\pr D$ and wherever the flow and the $\pr D$ intersect, they meet orthogonally. From a partial differential equation (PDE) point of view, the free boundary mean curvature flow is a nonlinear PDE subject to Neumann boundary conditions. For the study of free boundary mean curvature flow, we refer the readers to \cite{stahl1996regularity, stahl1996convergence, Edelen16_convexity-free-boundary-MCF, Buckland05_free-boundary-MCF, EdelenHaslhoferIvakiZhu22_MeanConvex-free-boundary-MCF}, among others.

The construction uses an interpolation argument. We will consider a (singular) foliation of rotationally symmetric surfaces $\{\Gamma_s\}_{s\in I}$ in the cylinder. Each surface consists of two planes with holes connected by a neck. As the size of the hole varies, the behavior of the surface under mean curvature flow changes: when the hole is very small, the surface pinches at the origin; when the hole is very large, the surface contracts to the boundary. Thus, there must be a critical surface $\Gamma{s_0}$ where neither pinching nor contracting to the boundary occurs under mean curvature flow. We prove that this critical surface indeed converges to the free boundary disk with multiplicity $2$ under the mean curvature flow.

This interpolation argument has been used in the construction of mean curvature flow, and we refer the readers to \cite{White02_ICM, ChuSun23_genus} for previous applications of this technique.

In the next step, we let $R\to\infty$ to extract a subsequence of the flows constructed in the first step, obtaining a limit flow. A key aspect of the proof is ensuring that our limit is a non-static flow, rather than a static minimal surface such as a catenoid. This requires a careful selection of the (singular) foliation in the first step. Additionally, we derive quantitative estimates of the obtained flow through the use of a barrier argument.

One technique that we frequently use in our proof is the parabolic Sturmian theorem proved by Angenent \cite{angenent88_zero-set}. Such a Sturmian theorem was previously used in the study of rotationally symmetric mean curvature flow by Altschuler-Angenent-Giga \cite{AAG}. Roughly speaking, the Sturmian theorem can be used to control the number of intersection points between a mean curvature flow (that we are interested in) and a barrier. This provides a local comparison of the flow with a barrier, which is more powerful than the classical comparison principle of mean curvature flows, as it typically yields a global comparison that may not be sufficiently strong. We believe that exploring the application of the parabolic Sturmian theorem can be useful in other scenarios.

\section{Preliminary}

\subsection{Rotationally symmetric surface} \
We begin with some basic settings and notions. By taking the quotient of an $SO(n-1)$ action, any rotationally symmetric surface in $\mathbb{R}^n$ can be represented by a curve (known as the profile curve) in the half-plane, denoted as $\{(x,y) \mid x \geq 0\}$. The $y$-axis serves as the rotational axis. The area of the surface can be described by the length of the profile curve in the following metric:
\begin{equation}
    g_{\text{rot}}:=x^{2(n-2)}(dx^2+dy^2).
\end{equation}

In this paper, we only consider the case $n=3$. 

Throughout this paper, we will study the rotationally symmetric surfaces that are also reflexive symmetric. Given any curve $\gamma$ in the region $\{(x,y) | x, y \geq 0\}$ that touches the $x$-axis, let $S(\gamma)$ be the rotationally symmetric hypersurface in $\R^3$ obtained by first reflecting $\gamma$ along the $x$-axis, then rotating $\gamma$ and its reflection together about the $y$-axis. For any continuous function $f(x)$, we denote the graph of $f(x)$ in $\{(x,y)|x\geq 0\}$ as $G_f$.

Given $R > 0$, the solid closed cylinder with radius $R$ will be denoted by $C_R$. After taking the quotient of the $SO(n-1)$ action, $C_R$ can be described as the
$
C_R = \{ (x,y) | x \leq R \}$.

Given $0 < p < R$, $\e \geq 0$, let $\mathcal{F}_{R,p,\e}$ denote the set of all continuous functions $f: [p, R] \to [0, \infty)$ with $f(p) = 0$, $f'(R) = 0$, which is smooth in the following sense:
\begin{equation} \label{smooth curve}
\begin{split}
&f \text{ is smooth on } (p, R], \text{ and f has a local inverse on } [p, p + \e], \text{ such that } \\ &u: [0, f(p + \e)] \to [p, p + \e],\ u(y) = f^{-1}(y) \text{ is a smooth function}.    
\end{split}
\end{equation}

Clearly, the graph of $f$ can be expressed as the union of two graphs of smooth functions, $u$ and $v$, where $u = f^{-1}$ as mentioned above, and $v: [p + \e, R] \to [0, \infty)$, $v(x) = f(x)$. We name $u$ as the {\bf vertical graph function}, and its graph as the {\bf vertical graph}, $v$ as the {\bf horizontal graph function}, and its graph as the {\bf horizontal graph}.

We remark that our terminology is different from \cite{AAG}, because we view the graph in the direction of the rotationally symmetric plane, while they view the graph in the direction of the rotational axis. So our vertical graph is the horizontal graph in \cite{AAG}, and our horizontal graph is the vertical graph in \cite{AAG}.

Let $\mathcal{F}_{R, p} = \bigcup_{\e\geq 0} \mathcal{F}_{R, p, \e}$, and for any $f \in \mathcal{F}_{R, p}$, we name the point $(p,0)$ as the {\bf neck point} of $f$. Let $\mathcal{F}_R = \bigcup_{p\in(0,R)} \mathcal{F}_{R, p}$, and $\mathcal{F} = \bigcup_{R > 0} \mathcal{F}_R$.

For any $f \in \mathcal{F}_R$, $f'(R) = 0$ implies the curve $G_f$ and the line $x = R$ intersect orthogonally, which guarantees that $S(G_f)$ is a free boundary surface.

Let $\mathcal{G}_{R, p}$ denote a subset of $\mathcal{F}_{R, p}$, containing all functions $f$ with the following derivative conditions:
\begin{equation} \label{derivative condition of f}
f'(x) > 0, \ \forall x \in (p, R).
\end{equation}

Then by the inverse function theorem, condition \eqref{derivative condition of f} is equivalent with 
the following derivative conditions of the vertical graph function $u$ and the horizontal graph function $v$:
\begin{equation} \label{derivative condition of u and v}
u'(y)>0,\  v'(x) > 0.
\end{equation}

Finally, $\mathcal{G}_R = \bigcup_{p \in (0,R)} \mathcal{G}_{R, p}$. This is the class of functions whose graphs will be the profile curves of the mean curvature flow we will study.

\begin{figure}[htb]
    \centering
    \begin{tikzpicture}
    \draw[->] (0,0) -- (13,0) node[below] {$x$};
    \draw[->] (0,0) -- (0,4) node[left] {$y$};
    \node[below left] at (0,0) {$O$};
    \draw (12,0) -- (12,4) node[right] {$x = R$};
    \draw plot[smooth,tension=.5]
coordinates {(2.7,0) (2.7,1) (2.8, 2) (3,2.6) (4, 3) (6,3.2) (10,3.4) (12,3.4)};
    \node[below left] at (2.7,1.7) {vertical graph};
    \node[below left] at (2,1) {$x = u(y)$};
    \node[below] at (9.5,3) {horizontal graph $y = v(x)$};
    \fill(2.7,0)circle(1.5pt)node[below]{neck point $(p,0)$};
    \node[below] at (12,0) {$(R,0)$};
    \draw (2.7,0.1) coordinate (a) -- (2.7,0) coordinate (b) -- (2.71,0) coordinate (c);
    \tkzMarkRightAngle (a,b,c)
     \draw (11.9,3.4) coordinate (d) -- (12,3.4) coordinate (e) -- (12,3.3) coordinate (f);
    \tkzMarkRightAngle (d,e,f)
    \end{tikzpicture}
    \caption{Example of a curve that we study.}
\end{figure}
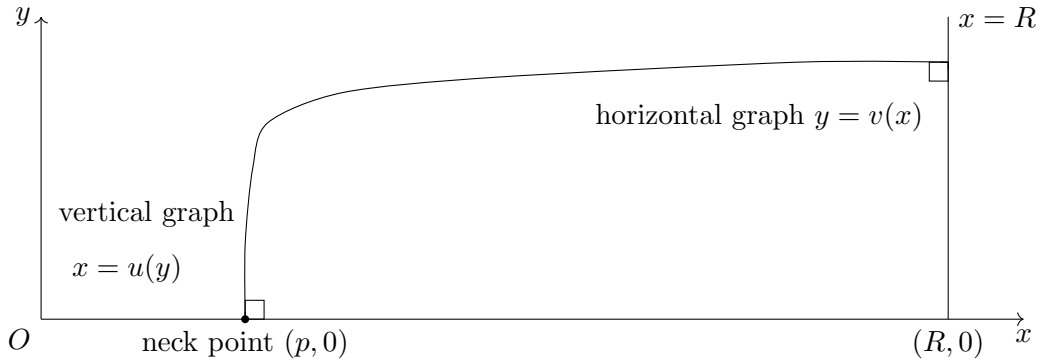

\subsection{Mean curvature flow for rotationally symmetric free boundary surface} \

Suppose $f_0 \in \mathcal{F}_R$, let $\Gamma = S(G_{f_0})$ be a rotationally symmetric free boundary surface in $C_R$. Then for $f(\cdot, t) \in \mathcal{F}_R$, $\Gamma_t = S(G_{f(\cdot, t)})$ is said to be a mean curvature flow of free boundary surfaces with initial condition $f_0$, if the corresponding vertical graph function, horizontal graph function $(u_0, v_0)$ of $f_0$, and $(u(\cdot, t), v(\cdot, t))$ of $f(\cdot, t)$ satisfy the following equation (see e.g. \cite{AAG}):

\begin{equation}
\label{mcf equation} 
\left\{
\begin{aligned}
&\frac{\partial u}{\partial t} = \frac{u_{yy}}{1+(u_y)^2} - \frac{1}{u}, \\
&\frac{\partial v}{\partial t} = \frac{v_{xx}}{1+(v_x)^2} + \frac{v_x}{x}, \\
&u(\cdot,0) = u_0,\  v(\cdot,0) = v_0, \\
&v_x(R, \cdot) = 0.
\end{aligned} 
\right.
\end{equation}

For simplicity, we will say $f(\cdot, t)$ is a solution to equation \eqref{mcf equation} if its corresponding vertical graph function and horizontal graph function solve this equation. We will call the first equation in \eqref{mcf equation} as vertical graph equation, and the second equation in \eqref{mcf equation} as horizontal graph equation.

We have the following results concerning the mean curvature flow of free boundary surfaces, which follow from Stahl \cite{stahl1996regularity}.

\begin{prop} [\cite{stahl1996regularity}] \label{long time existence}
Given $f_0 \in \mathcal{F}_R$, there exists a unique solution to equation \eqref{mcf equation} with free boundary condition on a maximal time interval $[0, T)$. This solution is smooth for $t > 0$, and in the class $C^{2+\omega, 1+ \omega/2}$ (with arbitrary $0 < \omega < 1$) for $t \geq 0$. Moreover, if $T < \infty$, then the curvature at some point in $G_{f(\cdot, t)}$ blows up as $t \to T$.
\end{prop}

Following from Proposition \ref{long time existence}, the first singular time only depends on the initial data $f_0$, and we will denote it by $T(f_0)$.

Suppose $\gamma_1$ and $\gamma_2$ are two curves in $\{x,y\geq 0\}\cap C_R$. We say that $\gamma_1$ is \textbf{on top of} $\gamma_2$, if for all pairs $(c, y_1, y_2)$ such that $(c,y_1) \in \gamma_1, (c,y_2) \in \gamma_2$, it holds that  $y_1 \geq y_2$. If $y_1 > y_2$ for all pairs $(c,y_1,y_2)$, then we say that $\gamma_1$ is \textbf{strictly on top of} $\gamma_2$. We have the following results based on the maximum principle.

\begin{prop} [Comparison principle] \label{comparison}
Suppose $\{S(\gamma^1_t)\}_{t\in[0,T(\gamma^1_0))}$ and $\{S(\gamma^2_t)\}_{t\in[0,T(\gamma^2_0))}$ are families of surfaces in $C_R$, evolving by mean curvature. If the initial curve $\gamma^1_0$ is on top of the initial curve $\gamma^2_0$, and the intersection angle between $\gamma^1_t$ and the line $\{ x=R\}$ is less than or equal to the intersection angle between $\gamma^2_t$ and the line $\{x = R\}$ (or the curve $\gamma^2_t$ does not intersect the line $\{x = R\}$). Then $\gamma^1_t$ is on top of $\gamma^2_t$ for all time $t \in [0, T)$, where $T = \min\{T(\gamma^1_0), T(\gamma^2_0)\}$, and the function $d(t) = \text{dist}(\gamma^1_t, \gamma^2_t)$ is monotone increasing.
\end{prop}

We refer the readers to \cite[Section 2.2]{Mantegazza_MCFbook} for a proof. Note that in \cite{Mantegazza_MCFbook}, the comparison principle was proved for mean curvature flow without boundary, however, the proof only uses the parabolic maximum principle, hence the proof also works for mean curvature flows with boundary. The boundary angle condition implies that the shortest distance between $\gamma_t^1$ and $\gamma_t^2$ is attained either when both points are in the interior of the domain, or when both points are attained on the boundary. This can be obtained by, for example, the first variational formula of the distance function. The first interior case is exactly the same as \cite[Section 2.2]{Mantegazza_MCFbook}. The latter case only happens if the intersection angle between $\gamma^1_t$ and the line $\{ x=R\}$ equals the intersection angle between $\gamma^2_t$ and the line $\{x = R\}$, and if we write the flow near the boundary as graphs, the parabolic maximum principle also implies that the distance is increasing.

\begin{rmk}
If $\gamma^1_0$ is on top of $\gamma^2_0$ and they are not the same curve, by strictly maximum principle, $\gamma^1_t$ will be strictly on top of $\gamma^2_t$ after a short amount of time.
\end{rmk}

We state below the continuous dependence theorem from \cite[Theorem 8.1]{Amann}. While Amann proved this result for general dynamics in fractional Sobolev spaces, we only require the basic case where solutions are smooth functions defined on open subsets of $[0,R] \times (0,+\infty)$. The assumptions of the theorem are satisfied since our equation (\ref{mcf equation}) is an autonomous quasilinear parabolic equation whose right-hand side is of class $C^1$. We refer readers to Sections $7$ and $8$ of \cite{Amann} for complete details.

\begin{prop} [Continuous dependence] \label{continuous dependence}
Suppose $f_t$ is a solution to equation \eqref{mcf equation}, with initial data $f_0 \in \mathcal{F}$. Write $\mathcal{D} = \{(t, f_0) \in \R_{\geq 0} \times \mathcal{F} | 0 \leq t < T(f_0)\}$. Then $\mathcal{D}$ is open in $\R_{\geq 0} \times \mathcal{F}$ and the map
\begin{align*}
\psi: \mathcal{D} \to \mathcal{F}, \qquad \psi(t, f_0) = f_t
\end{align*}
is a Lipschitz map.
\end{prop}

Later we will use rotationally symmetric minimal surfaces as barriers. The minimal surfaces are surfaces with zero mean curvature. They are static solutions to the mean curvature flow. Recall that the catenoids are rotationally symmetric minimal surfaces defined by $S(\gamma_c)$, where $\gamma_c$ is the curve
\begin{equation}
    \gamma_c(s)=(c\cosh(s/c),s), \ s\in[0,\infty).
\end{equation}

It is well-known that the catenoids and planes are the only rotationally symmetric minimal surfaces in $\R^3$. Furthermore, there is no catenoid with the $y$-axis as the axis of rotation that satisfies the free boundary condition in $C_R$. The following lemma is straightforward by the equation of catenoids.

\begin{lem} \label{limit of long time flow}

The only rotationally symmetric free boundary minimal surfaces in $C_R$ are the disks given by the intersection of $C_R$ and the plane $\{y = c\}$ for some constant $c$.

\end{lem}


\subsection{Three barriers} \

We will frequently use three known solutions of mean curvature flow as barriers. They are the planes, the catenoids, and the Angenent torus.

The planes are static mean curvature flows. If we write them as a rotational symmetric mean curvature flow, the profile curve is given by the curve $f(\cdot,t)\equiv C$ for a constant $C$.

The catenoids are also static mean curvature flows. If we write them as a rotational symmetric mean curvature flow, the horizontal graph function $u$ is given by
\[
u(y,t)=C\cosh(y/C), \ y \in [0, \infty),
\]
for a positive constant $C$. In addition, the vertical graph function of the catenoid is given by $v(x,t) = C \ln (\frac{x}{C} + \sqrt{\frac{x^2 - C^2}{C^2}})$, $x \in [C, \infty)$. It is worth noting that the catenoids satisfy the following properties. Firstly, for any fixed $C$, $v$ grows logarithmically, which implies that the catenoids can not be bounded between two planes. Secondly, as $C\to 0$, the catenoids converge to a plane with multiplicity $2$ (see Section \ref{SS:multiplicity} for the precise meaning).   

The Angenent torus was constructed by Angenent in \cite{Angenent92_Doughnut}. It is a rotationally symmetric torus in $\R^3$ that is self-shrinking under the mean curvature flow, and it shrinks to a point in finite time. Because the Angenent torus is rotationally symmetric, it is determined by the profile curve. We use $\mathcal{A}$ to denote the Angenent torus that is rotationally symmetric around the $y$-axis, with its closest point to the $y$-axis being $(11/10,0)$.

Although there is no explicit formula for the Angenent torus, there are some qualitative characterizations. The upper half of the profile curve (in the region $\{x\geq 0,\ y \geq 0 \}$) of $\mathcal{A}$ is confined within a rectangle defined by $1 < x < \alpha$, $0 \leq y < \alpha$, for some real number $\alpha > 2$. We use $T'$ to denote the time at which $\mathcal{A}$ shrinks to a single point under the mean curvature flow.

\subsection{Sturmian Theorem}

Briefly speaking, the parabolic Sturmian theorem asserts that the number of zeros of a solution to the parabolic PDE defined on $\R$ or a sub-interval of $\R$ can only decrease unless a new intersection point is produced on the boundary. We refer the readers to \cite[Theorem B, Theorem C]{angenent88_zero-set} for detailed statements, and \cite[Section 2]{angenent91_JDG-formation} for the adaption to nonlinear parabolic equations. Here we only state the applications of the parabolic Sturmian theorem to rotationally symmetric mean curvature flows.

\begin{lem}
    Suppose $f(\cdot,t)$, $g(\cdot,t)$ are two solutions to the rotationally symmetric mean curvature flow equation \eqref{mcf equation} defined on $C_R\times[0, T]$ ($T$ can be $+\infty$), $u(\cdot,t)$, $v(\cdot,t)$ are the corresponding vertical graph function, horizontal graph function of $f(\cdot,t)$, $\tilde u(\cdot,t)$ and $\tilde v(\cdot,t)$ are the corresponding vertical graph function, horizontal graph function of $g(\cdot,t)$. The number of intersection points of $f(\cdot,t)$ and $g(\cdot,t)$ is denoted by $z(t)$. Then we have the following properties:

    \begin{enumerate}
        \item $z(t)$ is a finite number for $t>0$, unless $f(\cdot,t)=g(\cdot,t)$.
        \item $z(t)$ is nonincreasing in $t$, unless $f(\cdot,t)$ and $g(\cdot,t)$ produce a new intersection point on $\partial C_R$.
        \item If $f(\cdot,t_0)$ and $g(\cdot,t_0)$ touches at a point $p$, in the sense that two curves intersect at $p$ at time $t_0$ and the unit tangent vector at $p$ coincides, then 
        \begin{itemize}
            \item either there exists $\delta>0$ such that $f(\cdot,t)$ and $g(\cdot,t)$ have at least two intersection points for $t\in(t_0-\delta,t_0)$ and at most one intersection point for $t\in(t_0,t_0+\delta)$,
            \item or there exists $\delta>0$ such that $f(\cdot,t)$ and $g(\cdot,t)$ have at most one intersection point for $t\in(t_0-\delta,t_0)$ and at least two intersection points for $t\in(t_0,t_0+\delta)$.
        \end{itemize}
    \end{enumerate}
\end{lem}

The proof is an application of Angenent's parabolic Sturmian theorem to $u$, $\tilde u$, and $v$, $\tilde v$. 

\subsection{Multiplicity}\label{SS:multiplicity}
Multiplicity is a concept in geometric measure theory. Roughly speaking, the multiplicity of a surface being $m$ simply means that we count the surface $m$ times. For example, consider the union of three surfaces $\{y=x^2/n\}$, $\{y=-x^2/n\}$ and $\{y=0\}$. Let $n\to\infty$, the limit is the plane $\{y=0\}$ with multiplicity $3$. For the detailed definition, we refer the readers to \cite{Simon83_GMT}.

Throughout this paper, the multiplicity $2$ convergence is understood as follows. We say a mean curvature flow $\{M(t)\}_{t\in[0,\infty)}$ converges to a plane with multiplicity $2$ if for any compact region $K$ and $\epsilon>0$, there exists $t_\epsilon>0$, such that for $t>t_\epsilon$, $[M(t)\cap K]\backslash B_\epsilon(0)$ can be written as the union of two graphs of functions $f_1(\cdot,t),f_2(\cdot,t)$ over the plane, and $\|f_1(\cdot,t)\|_{C^1}\to 0$, $\|f_2(\cdot,t)\|_{C^1}\to0$ as $t\to\infty$\footnote{Moreover, by Brakke's regularity theorem (see \cite{Brakke78, White05_local}), or Ecker-Huisken's estimate \cite{EckerHuisken91_interior}, this implies that $\|f_1(\cdot,t)\|_{C^l}\to 0$, $\|f_2(\cdot,t)\|_{C^l}\to0$ for any $l\in\Z_+$.}. This convergence is stronger than the higher multiplicity convergence in the sense of geometric measure theory.

We remark that the higher multiplicity convergence in the sense of geometric measure theory can be much more complicated than the pattern described above. Nevertheless, in this paper, any multiplicity $2$ convergence would be the above pattern.

\section{Free boundary cases}\label{S:Free boundary cases}

In this section, we study the solutions to the equation \eqref{mcf equation}. The main goal is to construct an immortal smooth free boundary mean curvature flow within the solid cylinder $C_R$, converging to the free boundary disk with multiplicity $2$ as time approaches $\infty$. The explicit statements are in Theorem \ref{free boundary immortal flow} and Theorem \ref{free boundary multiplicity 2}.

Let us sketch the idea of proof. First, we use the idea of Altschuler-Angenent-Giga \cite{AAG} in the study of rotationally symmetric mean curvature flow to show that the singularities of rotationally symmetric mean curvature flow with free boundary can only occur either on the rotational axis or on the free boundary. 

Next, we construct a family of rotational symmetric free boundary surfaces $\{M^s\}_{s\in[0,1]}$, such that under the MCF, $M^0(t)$ has a singularity on the axis while $M^1(t)$ has a singularity on the boundary. Then an interpolation argument shows that there exists $s_0\in[0,1]$, such that the MCF starting from $M^{s_0}$ is smooth for all future time.

By the first variational formula, $\int_{0}^\infty\int_{M^{s_0}(t)}|\vec H|^2 d\mathcal H^2 dt\leq \area(M^{s_0})<+\infty$. This implies that $M^{s_0}(t)$ subsequentially converges to a stationary varifold with free boundary. To show that the limit is a plane with multiplicity $2$, we need a gradient estimate, as stated in Proposition \ref{gradient estimate}. This estimate is inspired by \cite{AAG}. 

\subsection{Rotationally symmetric MCF with free boundary within a cylinder}
We will introduce the following notations throughout this and the next section. Suppose $f(\cdot, t)$ is a solution to the equation \eqref{mcf equation}, $u(\cdot, t)$, $v(\cdot, t)$ will denote the vertical graph function and the horizontal graph function of $f(\cdot, t)$.

\begin{prop} \label{consistent}
Suppose $\Gamma_t = S(G_{f(\cdot, t)})$, where $f(\cdot, t) \in \mathcal{F}_R$ is a solution to equation \eqref{mcf equation}. If the initial condition $f(\cdot, 0) \in \mathcal{G}_R$, and $T = T(f(\cdot, 0)) > 0$, then $f(\cdot, t) \in \mathcal{G}_R$ for all $t \in (0, T)$.
\end{prop}

\begin{proof}

Recall that $f(\cdot, 0) \in \mathcal{G}_R$ means $u(\cdot, 0)$, $v(\cdot, 0)$ satisfy the derivative condition \eqref{derivative condition of u and v}. Since $u(\cdot, t)$, $v(\cdot, t)$ are solutions to the equation \eqref{mcf equation}, we know that $u_y$ and
 $v_x$ satisfies the following equation
\begin{equation} \label{derivative pde}
\begin{aligned}
& \frac{\partial u_y}{\partial t} = \frac{(u_y)_{yy}}{1 + (u_y)^2} - 2\frac{u_y[(u_y)_y]^2}{(1+(u_y)^2)^2} + \frac{u_y}{u^2}, \\
& \frac{\partial v_x}{\partial t} = \frac{(v_x)_{xx}}{1+(v_x)^2} - 2\frac{v_x[(v_x)_x]^2}{(1+(v_x)^2)^2} + \frac{(v_x)_x}{x} - \frac{v_x}{x^2}.
\end{aligned}
\end{equation}

By the strict maximum principle of the quasilinear equation, we claim that for $t > 0$,
\begin{equation} \label{derivative estimate}
\begin{aligned}
&u_y(y, t)>0,\ v_x(x, t) > 0.
\end{aligned} 
\end{equation}
    
\end{proof}

\begin{remark}
    We note that our initial data may not be smooth but only Lipschitz later on. (see Section \ref{interpolation surfaces}). Nevertheless, we can approximate our initial data by smooth monotone increasing functions. Therefore we can still obtain the desired monotonicity for $t>0$.
\end{remark}

It follows from Proposition \ref{consistent} that, if $f(\cdot, 0) \in \mathcal{G}_R$, then $f(\cdot, t) \in \mathcal{G}_R$. By the derivative condition \eqref{derivative condition of f}, we know that 
\begin{align*}
f(R,t) = \max \{y |\ (x,y) \in G_{f(\cdot, t)} \text{ for some } x \in [0,R]\}.
\end{align*}
$f(R,t)$ can be viewed as the height of the curve $G_{f(\cdot, t)}$, and we prove in the next lemma that it is monotone.

\begin{lem} \label{height decreasing}
$f(R,t)$ is a decreasing function in $t$, and it is strictly decreasing if $f(\cdot,0)$ is not a constant function.
\end{lem}

\begin{proof}
The horizontal line $y = f(R,t)$ is on top of the curve $G_{f(\cdot, t)}$, and remains static under the mean curvature flow. By Proposition \ref{comparison}, for any $t < t'$,  $y = f(R,t)$ is on top of the curve $G_{f(\cdot, t')}$, thus $f(R, t') \leq f(R, t)$.

Moreover, if $f(\cdot, 0)$ is not a constant function, then the strong maximum principle (\cite[Theorem 3.1]{stahl1996convergence}) yields the result.

\end{proof}

Although Lemma \ref{height decreasing} shows that the height of the profile curve of the flow strictly decreases, we do not have a quantitative estimate of the decreasing rate. To obtain such a bound, we need to construct a new family of barriers.

Given $a \in (0, R)$, Consider a smooth function $l_a \in C^{\infty}([0,R])$ such that 
\begin{equation*}
l_a(x) = 
\left\{
\begin{aligned}
&0 \qquad \text{for} \qquad 0 \leq x \leq \frac{a}{2}, \\
&1 \qquad \text{for} \qquad a \leq x \leq R,
\end{aligned} 
\right.
\end{equation*}
and $l_a' (x) > 0$ for $x \in (\frac{a}{2}, a)$. Let $\Gamma^a$ be the hypersurface given by rotating the graph of $l_a$ along the $y-$axis, and denote the corresponding solution to the mean curvature flow equation by $\Gamma^a (t)$. Since the initial surface $\Gamma^a(0)=\Gamma^a$ is the graph of a Lipschitz function, it follows from the rotational symmetry and the results in \cite{EckerHuisken89_EntireGraph} that there exists a smooth solution $L_a(x,t)$ to the horizontal graph equation with initial data $L_a(x, 0) = l_a(x)$, and boundary condition $\frac{\partial}{\partial x} L_a (0, t) = \frac{\partial}{\partial x} L_a (R, t) = 0$, such that $\Gamma^a(t)$ is given by rotating the graph of $L_a(x,t)$ along the $y-$axis. 

The following lemma shows that after evolving for a sufficiently long time, the height of the $L_a(x,t)$ will decrease for a definite amount.

\begin{lem} \label{height barrier}

There exist a constant $\beta_{a, R} > 0$ and a time $T_{a, R} > 0$ such that $L_a(R, t) < 1 - \beta_{a,R}$ for all $t > T_{a, R}$.

\end{lem}

\begin{proof}
By the maximum principle, we have $0 < L_a(x,t) < 1$ for all $x \in (0, R), t > 0$. In addition, $\frac{\partial}{\partial x} L_a \geq 0$ is bounded. By a similar argument as in the proof of Lemma \ref{height decreasing}, we know that $L_a(R, t)$ is a strictly decreasing function of $t$. 
\end{proof}

Next, we will examine the behavior of the neck point when the singularity emerges during the mean curvature flow.

In the proof of Lemma \ref{vertical graph gradient upper bound}, Corollary \ref{height 0 neck R}, Lemma \ref{vertical graph gradient lower bound}, Proposition \ref{singular time}, we use the following settings and notations. Let $f(\cdot, t) \in \mathcal{F}_R$ be a solution to equation \eqref{mcf equation}, with the initial condition $f(\cdot, 0) \in \mathcal{G}_R$, and the first singular time is $T > 0$. By Proposition \ref{consistent}, we know that $f(\cdot,t)$ has a positive derivative with respect to $x$ except at the two endpoints. Then by the inverse function theorem, we can extend the vertical graph equation $u(\cdot, t)$ to be defined on $[0, f(R,t)]$, and $u(\cdot, t)$ is smooth on $[0, f(R,t))$, for $t \in (0, T)$.

We adapt the idea from \cite[Theorem 4.3]{AAG} to get the following gradient estimate.

\begin{lem}
\label{vertical graph gradient upper bound}
Given $T>0$, there exists a continuous nonincreasing function $\sigma : (0, \frac{R}{2}] \to \R_{+}$ that only depends on the neck point of the initial condition (i.e. $u(0,0)$) and $T$, such that 
\begin{align*}
0 < u_y(y,t) \leq e^{\sigma(\delta)/t}, \quad \delta = \min \{u(y,t), R - u(y,t)\},
\end{align*}
holds for all $0 < t < T$, $0 < y < f(R,t)$.

\end{lem}

\begin{proof}

Let $a = u(0,0)$. As shown in equation \eqref{derivative pde}, we know $p(x,t) = \frac{\partial}{\partial x} L_a(x,t)$ satisfies the linear parabolic equation
\begin{align*}
\frac{\partial p}{\partial t} = \frac{1}{1+p^2} p_{xx} - \left(\frac{2p p_x}{(1+p^2)^2} - \frac{1}{x}\right)p_x - \frac{1}{x^2} p.
\end{align*}

Since $p(x,0)$ is positive on $(\frac{a}{2},a)$, it follows from equation $(4.3)$ in the proof of Theorem $4.3(b)$ in \cite{AAG} that for every $\delta > 0$, there exists a constant $A_{\delta} < \infty$ such that
\begin{equation} \label{p estimate}
p(x,t) \geq e^{-A_{\delta}/t},   
\end{equation}
for all $\delta \leq x \leq R - \delta$ and all $0 < t < T$. We can choose the constant $A_\delta$ so that it is nonincreasing in $\delta$. 

Next, we translate $\Gamma^a(t)$ along the $y-$axis by $\xi$ to get a new mean curvature flow $\Gamma^a_{\xi}(t)$. Its profile curve is given by
\begin{align*}
y = L_a(x,t) + \xi.
\end{align*}

Since $\frac{\partial}{\partial x} L_a(x,t) > 0$ for all $t > 0$, for each $t > 0$, denote the inverse function of $L_a(x,t)$ by $w(y, t)$. Therefore we can also represent $\Gamma^a_{\xi}$ by the curve $x = w(y-\xi, t)$. Now we consider $0 < t_0 < T$ and $0 < y_0 < f(R, t_0)$. There is a unique $\xi \in \R$ with
\begin{equation} \label{wu relateion}
w(y_0 - \xi, t_0) = u(y_0, t_0).
\end{equation}

Let $U(t)$ be the union of the graph of $u(y, t)$ and its reflection with respect to the $x-$axis (the graph of $u(-y,t)$). By the Sturmian Theorem, $U(t)$ and the graph of $w(y-\xi, t)$ cannot have fewer intersections when $t < t_0$ than what they have when $t = t_0$.

As $t \downarrow 0$, the graph of $w(y- \xi, t)$ (i.e. $\Gamma^a_{\xi}(t)$) converges to $\Gamma^a_{\xi}(0)$. This hypersurface intersects $U(0)$ exactly once (since the neck point of $U(0)$ is $(a,0)$). Therefore, the graph of $w(y-\xi, t_0)$ and $U(t_0)$ only intersect once. 

We prove $u_y(y_0, t_0) \leq w_y(y_0 - \xi, t_0)$ by contradiction. If $u_y(y_0, t_0) > w_y(y_0 - \xi, t_0)$, then for $\e > 0$ small, $u(y_0 - \e, t_0) < w(y_0 - \e - \xi, t_0)$. By using the monotonicity of the functions $u,w$, we know that as the $y$ coordinate of $U(t_0)$ decreases from $y_0$ to $0$, the $x$ coordinate of $U(t_0)$ is at least $u(0,t_0)$, and $u(0,t_0) > 0$ (otherwise a singularity at the origin appears at time $t_0$). On the other hand, the $x$ coordinate of the graph of $w(y - \xi,t_0)$ decreases to $0$. By applying the intermediate value theorem, $U(t_0)$ and the graph of $w(y-\xi, t_0)$ intersect at least twice, which contradicts the fact that they have at most one intersection point.

\begin{figure}[htbp]
\centering
\begin{minipage}[t]{0.48\textwidth}
\centering
\begin{tikzpicture}
    \draw[->] (0,0) -- (5,0) node[below] {$x$};
    \draw[->] (0,-3) -- (0,3) node[left] {$y$};
    \node[below left] at (0,0) {$O$};
    \draw plot[smooth,tension=.55]
coordinates {(4,2.4) (3,2.3) (2,2.1) (1.3,1.5) (1,0) (1.3,-1.5) (2,-2.1) (3,-2.3) (4,-2.4)};
    \node[below] at (3,3) {$U(0)$};
    \draw (0,0.5) -- (0.5,0.5);
    \draw plot[smooth,tension=.2]
coordinates {(0.5,0.5) (0.6,0.55) (0.7,0.7) (0.8,1.3) (0.9,1.45) (1,1.5)};
    \draw (1,1.5) -- (4,1.5);
    \node[below] at (3,1.5) {$\Gamma^a_{\xi}(0)$};
    \draw (4,-3) -- (4,3) node[right] {$x = R$};
    \fill(1.3,1.5)circle(1.5pt);
    \end{tikzpicture}
\caption{Initial conditions of these flows.}
\end{minipage}
\begin{minipage}[t]{0.48\textwidth}
\centering
 \begin{tikzpicture}
    \draw[->] (0,0) -- (5,0) node[below] {$x$};
    \draw[->] (0,-3) -- (0,3) node[left] {$y$};
    \node[below left] at (0,0) {$O$};
    \draw (4,-3) -- (4,3) node[right] {$x = R$};
    \draw[domain = 0 : 180] plot ({4- 3.6*sin (\x)}, {1.4*cos (\x)});
    \node[below] at (3.5,1.4) {$U(t_0)$};
    \draw plot[smooth,tension=.4]
coordinates {(0,0.5) (1,0.6) (1.2,1) (1.5,1.2) (2,1.3)};
    \node[below] at (1.6,2)
    {$\Gamma^a_{\xi}(t_0)$};    \fill(1.15,0.85)circle(1.5pt)node[below right]{\scriptsize\( (u(y_0,t_0), y_0) \)};
    \fill(0.69,0.54)circle(1.5pt);
    \end{tikzpicture}
\caption{By the Sturmian theorem, the phenomenon in this figure cannot happen.}
\end{minipage}
\end{figure}

Let $\delta = \min \{u(y_0,t_0), R - u(y_0,t_0)\}$, by equation \eqref{p estimate}, \eqref{wu relateion}, and the inverse function theorem, $
u_y(y_0, t_0) \leq w_y(y_0 - \xi, t_0) \leq e^{A_{\delta}/t_0}$.

\end{proof}

\begin{rmk}
From the free boundary condition, we can see that $\sigma(\delta) \to \infty$ as $\delta \to 0$.
\end{rmk}

By Lemma \ref{height decreasing}, we know the limit of the height $h = \lim\limits_{t \to T} f(R,t)$ exists. The above gradient estimate implies that if the height of the function $f(\cdot, t)$ tends to $0$, then the neck point must tend to the boundary of the cylinder. In other words, the flow must shrink to a point on the boundary. 

\begin{coro} \label{height 0 neck R}

If $h = 0$, then $\lim\limits_{t \to T} u(0,t) = R$.

\end{coro}

\begin{proof}

We prove this by contradiction. Suppose not, then there exists $0 < \e < \frac{R}{2}$ and an increasing sequence $\{t_i\}$ such that $\lim\limits_{i \to \infty} t_i = T$, $\lim\limits_{i \to \infty} u(0,t_i) < R - \e$.

Then $f(x,t_i)$ is well-defined on $[R-\e, R - \frac{\e}{2}]$, let $a(t_i) = f(R -\e, t_i), b(t_i) = f(R - \frac{\e}{2}, t_i)$, then by Lemma \ref{vertical graph gradient upper bound},
\begin{align*}
\frac{\e}{2} = u(b(t_i),t_i) - u(a(t_i), t_i) = \int_{a(t_i)}^{b(t_i)} u_y (y, t_i) dy \leq (b(t_i) - a(t_i)) e^{\sigma(\frac{\e}{2})/t_i}.
\end{align*}

Hence $f(R,t_i) \geq b(t_i) \geq \frac{\e}{2} e^{-\sigma(\frac{\e}{2})/t_i} \geq \frac{\e}{2} e^{-\sigma(\frac{\e}{2})/T}$ for all $i$. Moreover $\lim\limits_{t_i\to T}f(R,t_i)\geq \frac{\e}{2} e^{-\sigma(\frac{\e}{2})/T}>0$, which contradicts to $\lim\limits_{t \to T} f(R,t) = 0$.

\end{proof}

On the other hand, if the limit of the height $h$ is not zero, we obtain an improved gradient estimate.

\begin{lem} \label{vertical graph gradient lower bound}
If $h > 0$, then for any $0 < a < h$, let $\lambda = \frac{\pi}{h - a}$, there exists a constant $\varepsilon > 0$ such that $u_y(y,t) \geq \varepsilon e^{-\lambda^2 t} \sin(\lambda (y-a))$ for all $y \in [a,h], 0 \leq t < T$. In addition, $u(a,t) \leq R - \frac{2\varepsilon}{\lambda} e^{-\lambda^2 T}$.
\end{lem}

\begin{proof}

For $0 \leq t < T$, $u(y,t)$ is well-defined on $[0,h]$ and $u_y(y,t) > 0$ for $y \in (0,h]$. We know $u$ satisfies the vertical graph equation
\begin{align*}
u_t = \frac{u_{yy}}{1+(u_y)^2} - \frac{1}{u}.
\end{align*}

Define $\theta (y,t) = \arctan u_y(y,t)$, then $\theta(y,t) \in (0, \frac{\pi}{2})$ for $y \in [a,h]$, and
\begin{equation*}
\begin{aligned}
& u_t = \theta_y - \frac{1}{u}, \qquad  \theta_t = \frac{1}{1 + (u_y)^2} (u_y)_t = \frac{1}{1 + (u_y)^2} (u_t)_y = \frac{1}{1 + (u_y)^2} \left(\theta_{yy} + \frac{u_y}{u^2}\right).
\end{aligned}
\end{equation*}

Hence $\theta_t - \frac{\theta_{yy}}{1 + (u_y)^2} > 0$ for $y \in [a,h]$. Since $\theta(y,0) > 0$ for all $y \in [a,h]$, let $\varepsilon = \min_{y \in [a,h]} \theta(y,0) > 0$, $\varphi(y,t) = \varepsilon e^{-\lambda^2 t} \sin (\lambda(y-a))$. Then $\varphi_t = \varphi_{yy}$, $\varphi_{yy} \leq 0$ on $[a,h]$, thus
\begin{align*}
\varphi_t - \frac{\varphi_{yy}}{1+(u_y)^2} = \varphi_{yy} \frac{(u_y)^2}{1+(u_y)^2} \leq 0.
\end{align*}

As a consequence
\begin{equation*}
\begin{aligned}
& \varphi_t - \frac{\varphi_{yy}}{1+(u_y)^2} < \theta_t - \frac{\theta_{yy}}{1+(u_y)^2}, \\
& \varphi (y, 0) \leq \varepsilon \leq \theta(y, 0), \\
& \varphi(a, t) = 0 < \theta(a,t), \quad \varphi(h,t) = 0 < \theta(h, t).
\end{aligned}
\end{equation*}

We apply the classical maximum principle to conclude that $\theta(y, t) \geq \varphi (y,t)$ for all $a \leq y \leq h, 0 \leq t < T$. Therefore
\begin{equation*}
\begin{aligned}
& u_y(y,t) = \tan \theta(y,t) \geq \theta(y,t) \geq \varphi(y,t) = \varepsilon e^{-\lambda^2 t} \sin(\lambda (y-a)), \\
& u(h,t) - u(a,t) = \int_a^h u_y(y,t) dy \geq \int_a^h \varepsilon e^{-\lambda^2 t} \sin(\lambda (y-a)) dy = \frac{2\epsilon}{\lambda} e^{-\lambda^2 t}.
\end{aligned}
\end{equation*}
This implies $u(a,t) \leq R - \frac{2\varepsilon}{\lambda} e^{-\lambda^2 T}$.

\end{proof}

Now we are ready to describe the possible singular behaviors of the rotationally symmetric flows that we are interested in. There is a trichotomy: either the flow remains smooth forever, or it has a neck singularity on the rotational axis, or it shrinks to a singularity on the boundary.

\begin{prop} \label{singular time}
Suppose $f(\cdot, t)$ is a solution to the equation \eqref{mcf equation} with initial condition $f(\cdot, 0) \in \mathcal{G}_R$. Then the flow first becomes singular at time $T$ if and only if $\lim\limits_{t \to T} u(0,t) = 0$ or $\lim\limits_{t \to T} u(0,t) = R$. In addition, if such $T$ doesn't exist, then the mean curvature flow exists for all future time.
\end{prop}

\begin{proof}

It is clear that if $\lim\limits_{t \to T} u(0,t) = 0$, then a neck pinch singularity appears at the origin, and if $\lim\limits_{t \to T} u(0,t) = R$, a singularity appears at the boundary. Now we assume neither of the above happens, and we want to show that $T$ is not a singular time. 

We prove by contradiction and assume that a singularity appears at time $T$. By Corollary \ref{height 0 neck R}, $h > 0$. We claim that there exists $\e_1 > 0$ such that $u(0,t) > \e_1$ for all $0 \leq t < T$. Otherwise there exists a sequence $\{t_i\}$ in $[0,T)$ such that $\lim\limits_{i \to \infty} u(0,t_i) = 0$. Up to extracting from a subsequence, we can assume $t_i$ converges to $T' \in [0, T]$. Since $u(0, T) \neq 0$, $T' < T$, and a singularity appears at the origin at time $T'$, which is a contradiction.

By a similar argument, we can also assume that $u(0,t) < R -\e_1$ for all $0 \leq t < T$.

For any $0 < a < h$, by Lemma \ref{vertical graph gradient lower bound}, we know $u(a, t) < R - \frac{2\varepsilon}{\lambda} e^{-\lambda^2 T}$. Let $\e = \min \{\e_1, \frac{2\varepsilon}{\lambda} e^{-\lambda^2 T}\}$, we know $\e < u(0,t) < u(a,t) < R - \e$. 

We claim that $u(y,t)$ is smooth at $(y,t) \in [0, a] \times [0,T]$. This follows from $u_y(0,t) = 0$, a priori estimate for $u_y$ in Lemma \ref{vertical graph gradient upper bound}, and hence (see \cite{ladyzhenskaia1968linear}) for all higher derivatives of $u$ in the interior. Therefore the singularity can only appear on the boundary, i.e. at $(R,h)$.

Now consider the horizontal graph function $v(x,t)$, which is defined for $R - \e_1 \leq x \leq R$, $0 < t < T$, and it is uniformly bounded by the height of the initial condition. This function is a solution of the horizontal graph equation, so the Evans-Spruck estimates (\cite[Corollary 5.3]{EvansSpruck92_III}, also see \cite[Page 303]{AAG}) imply that $\nabla v$ as well as all higher space derivatives of $v$ are uniformly bounded on the region $\{(x,t): R- \e_1/2 \leq x \leq R,\ T/2 < t < T\}$. Hence $v(x,t) \to v(x,T)$ uniformly in $R - \e_1/2 \leq x \leq R$ as $t \nearrow T$. We have also shown that $v_t(r,t)$ is uniformly bounded for $R- \e_1/2 \leq x \leq R$, $T/2 < t < T$, hence $(R,h) = (R, v(R,T))$ cannot be a singular point.


\end{proof}

In the following proposition, we show that the appearance of the neck singularity is an open condition.

\begin{prop} \label{left open}

Let $\mathscr L$ denote the set of function $f_0 \in \mathcal{G}_R$ such that the solution $f(\cdot, t)$ to the equation \eqref{mcf equation} with initial condition $f(\cdot, 0) = f_0$ becomes singular in finite time, and $\lim\limits_{t \to T} u(0,t) = 0$ for some $T>0$, then $\mathscr L$ is an open set in $\mathcal{G}_R$ with respect to the $C^1$ norm.

\end{prop}

\begin{proof}
Suppose $f_0\in \mathscr L$, and becomes singular as $t\to T$. It suffices to show that there exists $\epsilon>0$, such that for any $\hat f_0$ with $\|\hat{f}_0-f_0\|_{C^1}\leq \epsilon$, $\hat{f}_0\in \mathscr L$ as well.

Because $\lim\limits_{t \to T} u(0,t) = 0$, the mean curvature flow $\Gamma_t=S(G_{f(\cdot,t)})$ has a singularity at the origin. By \cite{AAG}, this is a neck singularity, i.e. $(T-t)^{-1/2}\Gamma_t$ converges to the cylinder $S^1(\sqrt{2})\times\R$ smoothly in $B_r$, for any $r>0$, as $t\to T$. In particular, there exists $t_0<T$, such that $(T-t_0)^{-1/2}\Gamma_{t_0}$ has distance at most $1/2$ away from $S^1(\sqrt{2})\times\R$, inside the ball $B_{10\alpha}$.

By the continuity of the mean curvature flow with respect to the initial data, for any $\delta>0$, there exists $\epsilon>0$, such that whenever $\|\hat{f}_0-f_0\|_{C^1}\leq \epsilon$, suppose $\hat f(\cdot, t)$ is the solution to the equation \eqref{mcf equation} with initial condition $\hat f(\cdot, 0) = \hat f_0$, $(T-t_0)^{-1/2}\hat \Gamma_{t_0}=(T-t_0)^{-1/2}S(G_{\hat f(\cdot,t_0)})$ has distance at most $1/2$ away from $(T-t_0)^{-1/2}\Gamma_{t_0}$. In particular, $(T-t_0)^{-1/2}\hat \Gamma_{t_0}$ has distance at most $1$ away from $S^1(\sqrt{2})\times\R$, inside the ball $B_{10\alpha}$.

Using an Angenent torus as a barrier, we see that $\hat \Gamma_{t}$ has a finite time singularity, which must occur on the rotational axis. This implies that $\hat{f}_0\in \mathscr L$.
\end{proof}

\subsection{Interpolation family of surfaces} \label{interpolation surfaces} \

We consider a family of initial data. Let 
\begin{equation*} 
    \rho_{\delta} = \left\{(\delta,y) \left| y \in \left[0,\frac{\alpha}{\delta}\right]\right.\right\} \cup \left\{\left.\left(x, \frac{\alpha}{\delta
}\right) \right| x \in [\delta, \infty)\right\},\ 
\delta \in (0, \infty),
\end{equation*} 
and let
\begin{equation} \label{initial curve definition}
    \rho_{\delta,R} = \rho_{\delta} \cap \{(x,y) | x \leq R\}
\end{equation}
for $R > 0$. For any $0 < \delta_1 < \delta_2 < R$, $\rho_{\delta_1, R}$ is on top of $\rho_{\delta_2, R}$, and they form a (singular) foliation.

\begin{rmk} 
For $\delta \in (0, R)$, even though our initial data $\rho_{\delta, R}$ are not contained in the set of graphs of functions in $\mathcal{G}_R$, they are locally Lipschitz. In fact, in \cite{stahl1996regularity}, Stahl obtained a local $C^1$ estimate of the free boundary mean curvature flow, see \cite[Remark 6.14]{stahl1996regularity}. Therefore, one can use an approximation argument to show that there exists a family of functions $f(t) \in \mathcal{G}_R, t \in (0, \e)$, for small $\e > 0$, such that $S(G_{f(t)})$ evolves by its mean curvature, and converges to $S(\rho_{\delta, R})$ as $t \to 0$. Such an approximation argument has been used by Ecker-Huisken in \cite[Theorem 4.2]{EckerHuisken91_interior} and we refer the readers to the discussions before \cite[Theorem 4.2]{EckerHuisken91_interior}. Therefore we can apply all arguments above to solutions with initial data $\rho_{\delta, R}$.
\end{rmk}

Let $f_{\delta}(\cdot, t)$ be the family of solutions to the equation \eqref{mcf equation} with initial data $\rho_{\delta, R}$. Denote the vertical graph function and the horizontal graph function of $f_{\delta}(\cdot, t)$ by $u_{\delta}(\cdot, t), v_{\delta}(\cdot, t)$. By Proposition \ref{long time existence}, we can write the first singular time of $f_{\delta}(\cdot, t)$ as $T(\rho_{\delta, R})$.

We need the following Lemmas to study the flows with a finite time singularity, and the flow that exists for all future time.

\begin{lem} \label{away from boundary}
Given $\delta > 1$, $R > 2 \alpha$. If $\lim\limits_{t \to T} u_{\delta}(0,t) = 0$, then $u_{\delta}(0,t) \leq \alpha  < R$ for all $t \in [0,T(\rho_{\delta, R}))$.

\end{lem}

\begin{proof}

Since $\delta > 1$, by Proposition \ref{comparison}, the horizontal line $y = \alpha$ is always on top of $G_{f_{\delta}(\cdot, t)}$.

We prove this lemma by contradiction. Suppose $u_{\delta}(0,t_1) > \alpha$ for some $t_1 \in [0,T(\rho_{\delta, R}))$. We choose $\e$ small so that $0 < \e < \min\{u_{\delta}(0,t_1) - \alpha, \frac{\alpha}{100}\}$. Consider the curve $\mathscr{C}: \{ (x,y) | x = (\alpha + \e)\cosh (\frac{y}{\alpha + \e}), y \geq 0 \}$ (the upper half profile curve of the catenoid with neck point $(\alpha + \e, 0)$). $\mathscr{C}$ intersect the boundary $x = R$ at $(R, (\alpha + \e)\ln (\frac{R}{\alpha + \e} + \sqrt{\frac{R^2 - (\alpha + \e)^2}{(\alpha + \e)^2}}))$. Since $R > 2\alpha$, and $\e < \frac{\alpha}{100}$, then $(\alpha + \e)\ln (\frac{R}{\alpha + \e} + \sqrt{\frac{R^2 - (\alpha + \e)^2}{(\alpha + \e)^2}}) > \alpha$. Thus $\mathscr{C}$ has exactly one intersection point with $\rho_{\delta, R}$.

By Sturmian Theorem, any new intersection point between $\mathscr{C}$ and $G_{f_{\delta}(x,t)}$ can only appear on the boundary. But we know for $t \in [0, T(\rho_{\delta, R}))$, $f_{\delta}(R,t) \leq \alpha < (\alpha + \e)\ln (\frac{R}{\alpha + \e} + \sqrt{\frac{R^2 - (\alpha + \e)^2}{(\alpha + \e)^2}})$, thus there will be no new intersection point appearing on the boundary.

Because $u_{\delta}(0, t_1) > \alpha + \e$, $f_{\delta}(R,t_1) < (\alpha + \e)\ln (\frac{R}{\alpha + \e} + \sqrt{\frac{R^2 - (\alpha + \e)^2}{(\alpha + \e)^2}})$, $G_{f_{\delta}(x,t_1)}$ and $\mathscr{C}$ can not intersect, otherwise they have at least two intersection points. Then $\mathscr{C}$ is on top of $G_{f_{\delta}(x,t_1)}$, which implies $u_{\delta}(0,t) \geq \alpha + \e$ for all $t \geq t_1$. This is a contradiction to $\lim\limits_{t \to T} u_{\delta}(0,t) = 0$.

\end{proof}

\begin{remark}
    In fact, using the parabolic Sturmian theorem to compare the flow with catenoids, one can show that the neck point moves monotonically to the origin after a sufficiently long time. We do not need this fact and hence we omit the proof here.
\end{remark}

We are now ready to proceed with the proof for the free boundary version of our main theorem.

\begin{thm} \label{free boundary immortal flow}
Given $R > 2 \alpha$, there exists an immortal rotationally symmetric free boundary mean curvature flow of surfaces in $C_R$, which is not a static plane.
\end{thm}

\begin{proof}

For $\delta \leq 1$, we know that the curve $\rho_{\delta, R}$ is on top of the profile curve of the Angenent torus $\mathcal{A}$. Since the neck point of $\mathcal{A}$ tends to the origin under the mean curvature flow at some finite time $T'$, by the comparison principle, $u_{\delta}(0, t) \to 0$ at some time $t < T'$. By the comparison principle Proposition \ref{comparison}, if $u_{\delta}(0, t) \to 0$ in finite time, then $u_{\delta'}(0,t) \to 0$ in finite time for all $0 < \delta' < \delta$.

For $\alpha \leq \delta < R$, the curve $\rho_{\delta, R}$ lies within the region $[\delta, R]\times[0,1]$. In contrast, the catenoid, described by $x = \cosh y$, remains static under the mean curvature flow and it lies strictly above the line $y = 1$ within the interval $[\delta, R]$ (since $\cosh 1 < 2 < \alpha \leq \delta$). Thus this catenoid is on top of $\rho_{\delta, R}$, and the function $y = \ln (x + \sqrt{x^2 - 1})$ has positive derivative at $x = R$. Thus by Proposition \ref{comparison}, this catenoid is on top of $G_{f_{\delta}(\cdot, t)}$ for all time $t$, which implies $u_{\delta}(0,t) \geq 1$ for all $t$.

Hence, by Proposition \ref{continuous dependence} and Proposition \ref{left open}, there exists a maximal interval $(0, n_R)$ such that for any $\delta$ within this interval, $u_{\delta}(0, t)$ converges to $0$ in finite time. As indicated in the preceding argument, $1 < n_R \leq \alpha$. By Proposition \ref{comparison} and Proposition \ref{singular time}, the singular time $T(\rho_{\delta, R})$ is strictly increasing in $\delta \in (0, n_R)$, and its limit as $\delta\to n_R$ has to be $\infty$, otherwise $u_{n_R}(0,t)$ will converge to $0$ in finite time by Proposition \ref{continuous dependence}.

By the selection of $n_R$, $u_{n_R}(0, t)$ never reaches $0$ in finite time. Combining this with Lemma \ref{away from boundary}, we conclude that for $\delta \in (1, n_R)$ and $t \in [0, T(\rho_{\delta, R}))$, we have $u_{\delta}(0, t) \leq 2 \alpha)$. Utilizing Proposition \ref{continuous dependence} and the fact that $\lim_{\delta \to n_R} T(\rho_{\delta, R}) = \infty$, we can deduce that $u_{n_R}(0, t) \leq 2 \alpha$ for all $t$.

As a result, $u_{n_R}(0,t)$ does not converge to either $0$ or $R$ within any finite time, by Proposition \ref{singular time}, the solutions $u_{n_R}(\cdot, t)$ and $v_{n_R}(\cdot, t)$ exist for all time $t \in [0, \infty)$. 
\end{proof}

From the construction above, we know that the solution $f_{n_R}(\cdot, t)$ to the equation \eqref{mcf equation} exists for all time $t \in [0, \infty)$. In the following, we will show the free boundary mean curvature flow induced from $f_{n_R}(\cdot,t)$ will converge to the plane with multiplicity $2$.

Next, we show that the neck point of the function $f_{n_R}$ converges to $0$ as $t \to \infty$. 

\begin{lem} \label{neck point limit}
$\lim\limits_{t \to \infty} u_{n_R} (0, t) = 0$.
\end{lem}

\begin{proof}

We prove this by contradiction. The idea is to compare the flow with the barrier that we constructed in Lemma \ref{height barrier}. Suppose there exists $a > 0$, and a sequence $\{t_i\}_{i\in\Z_+}, t_i \nearrow \infty$ such that $u_{n_R} (0, t_i) > a$. Up to extracting a subsequence, we may assume that $t_{i + 1} > t_i + T_{a, R}$, where $T_{a, R}$ is the constant in Lemma \ref{height barrier}.

Let $\chi(t)$ be the union of the graph of $f_{n_R}(\cdot, t)$ and its reflection with respect to the $x-$axis. $\chi(0)$ is bounded between the lines $y = \pm \alpha$. By the monotonicity of $f_{n_R}$ with respect to $x$, we know that the graph of $l_a(x) + f_{n_R}(R, t_i) - 1$ is on top of $\chi(t_i)$ for all $i$. Since $L_a(x, t) + f_{n_R}(R, t_i) - 1$ solves the mean curvature flow equation, by comparison principle Proposition \ref{comparison}, we know that the graph of $L_a(x, t) + f_{n_R}(R, t_i) - 1$ is on top of $\chi(t + t_i)$.

By Lemma \ref{height barrier}, 
\begin{align*}
f_{n_R}(R, t_{i + 1}) \leq L_a(x, t_{i+1} - t_i) + f_{n_R}(R, t_i) - 1 < f_{n_R}(R, t_i) - \beta_{a,R},
\end{align*}
for all $i$. Then $f_{n_R}(R, t_{i+1}) < f_{n_R}(R, t_1) - i \beta_{a,R}$. Let $i \to \infty$, we get a contradiction to the fact that $f_{n_R}(R, t) \geq 0$ for all $t$. Hence $\lim\limits_{t \to \infty} u_{n_R} (0, t) = 0$.

\end{proof}

\begin{rmk} \label{neck point upper bound}
Since $1 < n_R \leq \alpha$, and the flow $f_{n_R}(x,t)$ exists for all future time, by Lemma \ref{away from boundary} and Lemma \ref{neck point limit}, we have $u_{n_R}(0,t) \leq \alpha$ for all $R > 2\alpha$ and $t \geq 0$.
\end{rmk}

\begin{rmk} \label{immortal neck point}
The above argument works for all rotationally symmetric free boundary mean curvature flows $S(G_{f(\cdot, t)})$ that exist for all future time, which implies that their neck points must converge to $0$ as $t \to \infty$.

\end{rmk}

As a consequence, we can prove the following height lower bound which depends on the location of the neck point for the solution $f_{n_R}$.

\begin{coro} \label{lower bound on height}
Given $R > 2 \alpha$, if $u_{n_R}(0, t_0) \geq \kappa$ at time $t_0$, where $\kappa \in (0,R)$, then $f_{n_R}(R, t_0) > \frac{\kappa}{2}$. Conversely, if $f_{n_R}(R, t_0) \leq  \kappa$ at time $t_0$, then $u_{n_R}(0, t_0) < 2 \kappa$.
\end{coro}

\begin{proof}

Suppose $u_{n_R}(0,t_0) \geq \kappa$, we prove $f_{n_R}(R, t_0) > \frac{\kappa}{2}$ by contradiction. Suppose by contrary that $f_{n_R}(R,t_0) \leq \frac{\kappa}{2}$, then by Proposition \ref{consistent}, we know the graph of $f_{n_R}(\cdot, t_0)$ is bounded in the region $\{(x,y)|\ \kappa < x \leq R,\ 0 \leq y \leq \frac{\kappa}{2}\}$, and intersects the line $x = R$ orthogonally. Then consider the restricted curve $x = \frac{\kappa}{2} \cosh \frac{2y}{\kappa}, x \leq R$, which is on top of graph of $f_{n_R}(\cdot, t_0)$, and has smaller intersection angle with the line $x = R$. By Proposition \ref{comparison}, we know that $u_{n_R}(0,t) > \frac{\kappa}{2}$ for time $t \geq t_0$. However, By Lemma \ref{neck point limit}, we know that $u(0,t)\to 0$ as $t \to \infty$. This is a contradiction.

\end{proof}

We can also use the Angenent torus as a barrier to obtain a height upper bound.

\begin{lem} \label{height upper bound}
$f_{n_R}(x, t) < \alpha x$ for all $t \in [0,\infty), x \in (u_{n_R}(0,t),\frac{R}{\alpha}]$.
\end{lem}

\begin{proof}

We prove this by contradiction. Suppose not, then $f_{n_R}(x_0,t) \geq \alpha x_0$ for some $t \in [0,\infty)$, $x_0 \in (u_{n_R}(0,t),\frac{R}{\alpha}]$, then $G_{f_{n_R}(\cdot, t)}$ is on top of the Angenent torus $x_0 \mathcal{A}$, which implies that $f_{n_R}$ has a finite time singularities. This is a contradiction.
\end{proof}

We are ready to prove the key long-time gradient estimate of $f_{n_R}$.

\begin{prop} \label{gradient estimate}
For any $R > 2 \alpha$, given $0 < a < b < R$, there exist a constant $\omega$ depending on $a$, $R$ and a constant $T$ depending on $a$, $b$, $R$, such that $$\frac{\partial}{\partial x} f_{n_R}(x,t) \leq \tan \left(\frac{\pi}{2} \frac{\ln R - \ln b}{\ln R - \ln a} + \frac{\omega}{\ln t}\right),$$ for all $x \in [b,R], t \geq T$. In addition, $\lim\limits_{t \to \infty} f_{n_R}(R,t) - f_{n_R}(b,t) = 0$.
\end{prop}

\begin{proof}

Given any $0 < a < b < R$, by Lemma \ref{neck point limit}, we know that there exists $T > 10$ such that for any $t \geq T$, $u_{n_R}(0,t) < \frac{a}{2}$. Then $f_{n_R}(x,t)$ is a smooth function over $[a, R]$ for all $t \geq T$. For simplicity, we will use $f(x, t)$ to express the function $f_{n_R}(x,t)$ restricted on the interval $[a, R]$, which is smooth for all time $t \geq T$. We know
\begin{align*}
f_t = \frac{f_{xx}}{1+f_x^2} + \frac{f_x}{x}, \qquad f_x > 0 \text{ for all } x \in [a,R).
\end{align*}

Let $\phi(x,t) = \arctan (f_x(x,t)), x \in [a,R], t \in [T, \infty)$. Then $0 \leq \phi(x,t) < \frac{\pi}{2}$, and $f_x(x,t) = \tan (\phi(x,t))$. Then
\begin{align*}
\phi_x(x,t) = \frac{f_{xx}(x,t)}{1+f_x^2(x,t)}, \qquad f_t = \phi_x + \frac{f_x}{x}.
\end{align*}

Moreover, $$\phi_t = \frac{1}{1+f_x^2} (f_x)_t = \frac{1}{1+f_x^2} (f_t)_x = \frac{1}{1+f_x^2} (\phi_{xx} + \frac{f_{xx}}{x} - \frac{f_x}{x^2}).$$ 

Therefore,

\begin{equation} \label{phi bound}
\phi_t - \frac{\phi_x}{x} - \frac{1}{1+f_x^2} \phi_{xx} = - \frac{f_x}{x^2(1+f_x^2)} < 0.
\end{equation}

Let $\mu = \frac{\pi}{2 \ln \frac{R}{a}}$, and $\omega$ be a constant to be determined later. Let
\begin{align*}
\varphi(x,t) = \mu \ln \frac{R}{x} + \frac{\omega - x}{\ln t}.
\end{align*}

Then
\begin{equation*}
\begin{aligned}
&\varphi_x = -\frac{\mu}{x} - \frac{1}{\ln t} < 0, \quad \varphi_{xx} = \frac{\mu}{x^2} > 0, \\
&\varphi_t - \frac{\varphi_x}{x} - \frac{\varphi_{xx}}{1+f_x^2} \geq \varphi_t - \frac{\varphi_x}{x} - \varphi_{xx} = \frac{t \ln t - (\omega - x) x}{x t (\ln t)^2} \geq \frac{t \ln t - \omega R}{x t (\ln t)^2}.
\end{aligned}
\end{equation*}

Let $T' = T + R(R + \frac{\pi}{2})$, $\omega = R + \frac{\pi}{2} \ln T'$. Then 
\begin{align*}
t \ln t - \omega R \geq T' \ln T' - R(R + \frac{\pi}{2} \ln T') \geq R^2 (\ln T' - 1) > 0.  
\end{align*}

Thus for $x \in [a,R]$, $t \geq T'$, we know that
\begin{align*}
\begin{aligned}
& \varphi(x,T') \geq \frac{\omega - x}{\ln T'} \geq \frac{\pi}{2} \geq \phi(x,T'), \\
&\varphi_t - \frac{\varphi_x}{x} - \frac{\varphi_{xx}}{1+f_x^2} \geq 0 > \phi_t - \frac{\phi_x}{x} - \frac{\phi_{xx}}{1+f_x^2} , \\
&\varphi(a,t) \geq \mu \ln \frac{R}{a} = \frac{\pi}{2} \geq \phi(a,t), \\
&\varphi(R,t) > 0 = \phi(R,t).
\end{aligned}
\end{align*}

Now we can apply the parabolic maximum principle to conclude that $\varphi (x,t) \geq \phi(x,t)$ for all $x \in [a, R]$, $ t \geq T'$. On the other hand, for any $x \in [b, R]$, $\varphi(x,t) \leq \frac{\pi}{2} \frac{\ln R - \ln b}{\ln R -\ln a} + \frac{\omega - b}{\ln t}$. Thus there exists $T'' > T'$ such that for all $x \in [b, R]$ and $t \geq T''$, $\varphi(x,t) < \frac{\pi}{2}$. In summary, for $x \in [b,R]$, $t \geq T''$, we have the gradient estimate
\begin{equation}
    f_x(x,t) = \tan (\phi(x,t)) \leq \tan (\varphi(x,t)) \leq \tan (\varphi (b , t)) = \tan \left(\frac{\pi}{2} \frac{\ln R - \ln b}{\ln R -\ln a} + \frac{\omega - b}{\ln t}\right).
\end{equation}
Next, integrating this gradient bound yields
\begin{equation*}
f(R,t) - f(b,t) = \int_b^R f_x(x,t) dx \leq (R - b) \tan \left(\frac{\pi}{2} \frac{\ln R - \ln b}{\ln R -\ln a} + \frac{\omega - b}{\ln t}\right).
\end{equation*}

Let $t \to \infty$, 
\begin{align*}
0 \leq \limsup\limits_{t \to \infty} [f(R,t) - f(b,t)] \leq (R - b) \tan \left(\frac{\pi}{2} \frac{\ln R - \ln b}{\ln R -\ln a}\right).
\end{align*}

Finally let $a \to 0$, we have $
\lim\limits_{t \to \infty} [f(R,t) - f(b,t)] = 0$.

\end{proof}

The gradient estimate yields the following two corollaries, which can be combined to show that $f_{n_R}(x,t)$ will $C^1$ converge to $0$.

\begin{coro} \label{f uniform convergence}
$f_{n_R}(x, t)$ converges to $0$ uniformly as $t \to \infty$, for all $x \in (0,R]$.
\end{coro}

\begin{proof}

By Lemma \ref{height decreasing}, we know $\lim\limits_{t \to \infty} f_{n_R}(R,t)$ exists. By Lemma \ref{height upper bound}, Proposition \ref{gradient estimate}, for any $0 < b \leq \frac{R}{\alpha}$,
\begin{align*}
0 \leq \lim\limits_{t \to \infty} f_{n_R}(R,t) = \lim\limits_{t \to \infty} f_{n_R}(b,t) \leq \alpha b.
\end{align*}

Let $b \to 0$, we get $\lim\limits_{t \to \infty} f_{n_R}(R,t) = 0$. Since $f_{n_R}(\cdot, t)$ is a strictly increasing function, thus $f_{n_R}(x,t)$ converges to $0$ as $t \to \infty$ and this convergence is uniform in $x$. 

\end{proof}

\begin{coro} \label{fx uniformly convergence}
Given $0 < b < R$, $\frac{\partial}{\partial x} f_{n_R}(x,t)$ converges to $0$ uniformly as $t \to \infty$, for all $x \in [b, R]$.
\end{coro}

\begin{proof}

For any $\e > 0$, there exists $0 < a < b$ such that $\tan (\frac{\pi}{2} \frac{\ln R - \ln b}{\ln R - \ln a}) < \frac{\e}{2}$. Let $\omega$, $T$ be as in Proposition \ref{gradient estimate}. Then there exists $T' \geq T$ such that for $t \geq T'$, $\tan (\frac{\pi}{2} \frac{\ln R - \ln b}{\ln R - \ln a} + \frac{\omega}{\ln t}) < \e$. By Proposition \ref{gradient estimate}, for $x \in [b,R]$, $t \geq T'$,
\begin{align*}
\frac{\partial}{\partial x} f_{n_R}(x,t) \leq \tan \left(\frac{\pi}{2} \frac{\ln R - \ln b}{\ln R - \ln a} + \frac{\omega}{\ln t}\right) < \e.
\end{align*}

\end{proof}

\begin{thm} \label{free boundary multiplicity 2}
For any $R>2\alpha$, $S(G_{f_{n_R}(\cdot, t)})$ is a rotationally symmetric free boundary mean curvature flow in $C_{R}$, and the flow exists for all future time. In addition, the forward limit of this flow is the free boundary disk $\{y = 0\}\cap C_R$ with multiplicity $2$.
\end{thm}

\begin{proof}

Following from the proof of Theorem \ref{free boundary immortal flow}, we only need to prove that the forward limit of $S(G_{f_{n_R}(\cdot, t)})$ is the free boundary disk $\{y = 0\}\cap C_R$ with multiplicity $2$. In other words, for any $\e > 0$, $f_{n_R}(x,t)|_{x\in[\e,R]}$ converges to $0$ in $C^1$.

Given $\e > 0$, by Lemma \ref{neck point limit}, we know that there exists $T > 0$ such that the neck point of $f_{n_R}(\cdot, t)$ is contained in $B_{\frac{\e}{2\alpha}}(0)$ for $t \geq T$. So $f_{n_R}(x, t)|_{[\frac{\e}{2\alpha}, R]}$ is well defined for all $t \geq T$. By Lemma \ref{height upper bound}, $f_{n_R}(\cdot, t) |_{[0, \frac{\e}{2\alpha}]}$ is contained in $B_{\e}(0)$ for $t \geq T$. By Corollary \ref{f uniform convergence} and \ref{fx uniformly convergence}, we know $f_{n_R}(\cdot, t) |_{[\frac{\e}{2\alpha}, R]}$ $C^1$ converges to $0$ as $t \to \infty$. This completes the proof.

\end{proof}

\section{Passing to limit in \texorpdfstring{$\R^3$}{Lg}} \label{main result}

For simplicity, we slightly modify the notation. Let $f_R(x,t)$ denote the solution in $C_R$ that we obtained in the last section, and at time $t$, its graph is denoted by $\Gamma_{R,t}$. Let $u_R(x,t)$ be the vertical graph function of $f_R(x,t)$, and therefore the neck point of $f_R(x,t)$ is $u_R(0,t)$.

In this section, we let $R\to\infty$ to get a limit of the solutions that we have obtained in Section \ref{S:Free boundary cases}. There are two steps: first, we show that we can indeed find a limit mean curvature flow as $R\to\infty$. Second, we show that the limit mean curvature flow converges to the plane with multiplicity $2$. 

We start with some basic relationships between the solutions $f_R$ with different $R$. Recall that $n_R$ is the neck point value of the initial data of $f_R$, and the initial curve $\Gamma_{R,0}$ is $\rho_{n_R, R}$.

\begin{lem} \label{neck point comparison}
$n_r \leq n_R$ for any $2 \alpha  < r \leq R$.
\end{lem}

\begin{proof}

We argue by contradiction. If there exists $2 \alpha  < r \leq R$ such that $n_r > n_R$, then $\rho_{n_R, R}$ is strictly on top of $\rho_{n_r, r}$. By the derivative estimate \eqref{derivative estimate}, the intersection angle between $\Gamma_{R,t}$ and the line $x = r$ is less than $\frac{\pi}{2}$. 

Consider the curve $\Gamma_{r, t}$, and the curve $\Gamma_{R,t}$ restricted to the region $x \leq r$, by Proposition \ref{comparison}, the distance between these two curves under the mean curvature flow is monotonically increasing, therefore their neck points can not both converge to $0$ as $t\to\infty$

\end{proof}

Using the fact $1 < n_R \leq \alpha$ for all $R > 2 \alpha$ and combining with Lemma \ref{neck point comparison}, we get the following corollary which shows that the initial curves $\rho_{n_R, R}$ has a limit as $R\to\infty$.

\begin{coro} \label{initial neck point limit}
There exists $\eta \in (1, \alpha]$ such that $\lim\limits_{R \to \infty} n_R = \eta$. Therefore the initial curves $\rho_{n_R,R}$ converges to the curve $\rho_{\eta}$ as $R \to \infty$.
\end{coro}

Next, to take a (subsequential) limit as $R \to \infty$, we need some uniform gradient estimates of $f_{R}$. In the following proposition, we use catenoids as barriers to derive a uniform gradient estimate of $f_{R}$ away from the neck point.

\begin{prop} \label{large x uniform gradient estiamte}

For any $R > 2\alpha + 2$, $x \in [\alpha + 2, R]$, and any time $t \geq 0$, we have
\begin{equation} \label{horizontal uniform gradient bound}
\frac{\partial}{\partial x} f_R(x,t) \leq \frac{\alpha + 1}{\sqrt{x^2 - (\alpha + 1)^2}}.
\end{equation}

\end{prop}

\begin{proof}

Consider the catenoid with neck point $(\alpha + 1, \xi)$, namely we move the catenoid with neck point $(\alpha+1,0)$ upward with distance $\xi$. It is static under the mean curvature flow, and its profile curve $\overline{\mathcal{C}}_{\xi}$ is given by $x = (\alpha + 1) \cosh (\frac{y - \xi}{\alpha + 1})$. Let $\mathcal{C}_{\xi}$ denote the graph of $g_{\xi}(x) = \xi + (\alpha + 1) \ln (\frac{x}{\alpha + 1} + \sqrt{\frac{x^2 - (\alpha + 1)^2}{(\alpha + 1)^2}})$, $x \in [\alpha +1, \infty)$. Then $\overline{\mathcal{C}}_{\xi}$ is the union of $\mathcal{C}_{\xi}$ and its reflection with respect to the line $y = \xi$.

Now we compare the catenoids with the mean curvature flow $\Gamma_{R,t}$. Let $c = g_0(R) > g_0(2\alpha +2) \geq \alpha + 1$. The $y-$coordinate of the intersection point between $\mathcal{C}_{\xi}$ and the boundary of the cylinder $\{x = R\}$ is $\xi + c$. 

Recall that $\Gamma_{R,0} = \rho_{n_R,R}$ with $n_R \leq \alpha$, while the $x-$coordinates of any point on $\mathcal{C}_{\xi}$ is at least $\alpha + 1$. From the construction of $\rho_{\delta, R}$ \eqref{initial curve definition}, we observe that for $\xi \geq f_R(R,0) - c$, there is at most one intersection point between $\mathcal{C}_{\xi}$ and $\Gamma_{R,0}$; for $\xi <  f_R(R,0) - c$, there is no intersection point between $\Gamma_{R,0}$ and $\mathcal{C}_{\xi}$. 

We make the following claims:

\noindent\textbf{Claim 1.} For any $t$, there is at most one intersection point between $\Gamma_{R,t}$ and $\mathcal{C}_{\xi}$.

\noindent \emph{Proof.} By Sturmian Theorem, under the mean curvature flow, the intersection point between $\Gamma_{R,t}$ and $\mathcal{C}_{\xi}$ could only appear on the boundary $x = R$. As mentioned above, the boundary point of $\mathcal{C}_{\xi}$ is $(R, \xi + c)$.

If $\xi \geq f_R(R,0) - c$, then for any time $t > 0$, $f_R(R,t) < f_R(R,0) \leq \xi + c$, there will be no extra intersection point appeared on the boundary.

If $\xi <  f_R(R,0) - c$, by Lemma \ref{height decreasing}, the boundary point of $\Gamma_{R,t}$ moves toward $(R,0)$ monotonically under the flow, thus there will be at most one extra intersection point. \hfill\qed

\smallskip

Given time $t_1 \geq 0$, let $D(x) = g_0(x) - f_R(x,t_1)$, which is defined on $[\alpha + 1, R ]$, then $D'(R) = g_0'(R) > 0$. Since $D'(x)$ is a smooth function, there exists a maximal interval $[p, R]$ such that $D'(x)$ is nonnegative on this interval. Then
\begin{align*}
d := D(R) - D(p) = \int_p^R D'(x) dx > 0. 
\end{align*}

\noindent\textbf{Claim 2.} The lower bound of this maximal interval $p \leq \alpha + 2$.

\noindent \emph{Proof.} We argue by contradiction. If $p > \alpha + 2$, by the choice of $p$, there exists $\e_1 > 0$ small enough such that $D'(p-\e_1) < 0$, and $|D(p) - D(p-\e_1)| < \frac{d}{2}$, which implies $\frac{\partial}{\partial x} f_R(p - \e_1,t_1) > g_0'(p - \e_1)$ and $D(p- \e_1) < D(R)$.

let $\xi = - D(p - \e_1)$, then $\mathcal{\mathcal{C}}_{\xi}$ and $\Gamma_{R,t_1}$ intersect at $(p - \e_1, f_R(p - \e_1,t_1))$. Since $\frac{\partial}{\partial x} f_R(p - \e_1,t_1) > g_0'(p - \e_1)$, for small $\e_2 > 0$, we have $f_R(p - \e_1 + \e_2, t_1) > g_i(p - \e_1 + \e_2)$. While on the boundary, we have
\begin{align*}
g_i(R) = \xi + c = -D(p - \e_1) + g_0(R) > -D(R) + g_0(R) = f_R(R,t_1).
\end{align*}

\begin{figure}[htbp]
\centering
\begin{minipage}[t]{0.48\textwidth}
\centering
\begin{tikzpicture}
    \draw[->] (0,0) -- (6,0) node[below] {$x$};
    \draw[->] (0,0) -- (0,3) node[left] {$y$};
    \node[below left] at (0,0) {$O$};
    \draw (1,0) -- (1,2) ;
    \draw (1,2) -- (5,2) ;
    \node[below] at (3,2.7) {$\Gamma_{R,0}$};
    \draw[domain = 0 : 2.65] plot ({1.2* cosh (0.833*\x)}, {\x- 1});
    \node[below] at (4,1.1) {$\mathcal{C}_{\xi}$};
    \draw (5,0) -- (5,3) node[right] {$x = R$};
    \end{tikzpicture}
\caption{Initial status of the curves in the proof of Proposition \ref{large x uniform gradient estiamte}.}
\end{minipage}
\begin{minipage}[t]{0.48\textwidth}
\centering
 \begin{tikzpicture}
    \draw[->] (0,0) -- (6,0) node[below] {$x$};
    \draw[->] (0,0) -- (0,3) node[left] {$y$};
    \node[below left] at (0,0) {$O$};
    \draw (5,0) -- (5,3) node[right] {$x = R$};
    \draw[domain = 0 : 2.65] plot ({1.2* cosh (0.833*\x)}, {\x- 1});
    \node[below] at (2.1,1.5) {$\Gamma_{R,t}$};
    \draw plot[smooth,tension=.4]
coordinates {(2.2, 0.3) (2.6, 0.7) (2.7,0.9) (3.1, 1.1) (4,1.2) (4.5, 1.23)(5,1.24)};
    \node[below] at (4.6,2.2)    {$\mathcal{C}_{\xi}$};
    \fill(3.8,1.18)circle(1.5pt);
    \fill(2.59,0.67)circle(1.5pt);
    \draw[dotted] (3.8,1.18) -- (3.8,0);
    \draw[dotted] (2.59,0.67) -- (2.59,0);
    \fill(3.8,0)circle(1.5pt);
    \fill(2.59,0)circle(1.5pt);
    \node[below] at (2.4,0) {$p - \e_1$};
    \node[below] at (4.1,0) {$p - \e_1 + \e_2$};
    \end{tikzpicture}
\caption{By the Sturmian theorem, the phenomenon in this figure can not happen.}
\end{minipage}
\end{figure}

By the intermediate value theorem, there exists $x_1 \in (p-\e_1 + \e_2, R)$ such that $g_\xi(x_1) = f_R(x_1, t_1)$. Therefore we have two intersection points $(p -\e_1, f_R(p-\e_1, t_1)), (x_1, f_R(x_1, t_1))$ between $\Gamma_{R,t_1}$ and $\mathcal{C}_{\xi}$, which contradicts to Claim $1$. \hfill\qed

\smallskip

By Claim $2$, for any $x \in [\alpha + 2, R]$, $D'(x) \geq 0$. This proves \eqref{horizontal uniform gradient bound}.

\end{proof}

Next, we prove a uniform gradient estimate of the vertical part of $f_R$ near the neck point. First of all, we show the neck point of $\Gamma_{R,t}$ is nondecreasing in $R$.

\begin{prop}\label{prop：neck comparison for different R}
The distance from the origin to the neck point of $\Gamma_{R,t}$ is nondecreasing in $R$, i.e. $u_{R_1}(0,t) \leq u_{R_2}(0,t)$ for all $2\alpha<R_1 < R_2$, $t \geq 0$.
\end{prop}

\begin{proof}

Given $R_1 < R_2$, we use $\Gamma'_{R_2,t}$ to denote the curve $\Gamma_{R_2,t}$ restricted on the region $x \leq R_1$. We know $\Gamma_{R_1,0}$ is on top of $\Gamma'_{R_2,0}$. For any $t \geq 0$, $\frac{\partial}{\partial x} f_{R_1}(R_1, t) = 0, \frac{\partial}{\partial x} f_{R_2}(R_1, t) > 0$.

From Lemma \ref{neck point comparison}, the neck points of the initial conditions have the following comparison: $u_{R_1}(0,0) \leq u_{R_2}(0,0)$. Now we have two cases: either this inequality is a strict inequality, or $u_{R_1}(0,0) = u_{R_2}(0,0)$.

\smallskip

\emph{Case} $1$. If $u_{R_1}(0,0) < u_{R_2}(0,0)$, then at $t = 0$, there is no intersection point between $\Gamma_{R_1,0}$ and $\Gamma_{R_2,0}$. By the Sturmian Theorem, either $\Gamma_{R_1,t}$ and $\Gamma_{R_2,t}$ has no intersection point, or there is a new intersection point that can only appear on the boundary.

If $f_{R_1}(R_1,t) > f_{R_2}(R_1,t)$ for all time $t \geq 0$, then there will be no new intersection point appearing on the boundary. Therefore for any $t \geq 0$,  $\Gamma_{R_1,t}$ and $\Gamma'_{R_2,t}$ do not intersect, which implies that $\Gamma_{R_1,t}$ is on top of $\Gamma'_{R_2,t}$, and $u_{R_1}(0,t) < u_{R_2}(0,t)$.

If $f_{R_1}(R_1,t) \leq f_{R_2}(R_1,t)$ at some time $t$, let $t_1$ be the first time where this inequality holds. Since $\frac{\partial}{\partial x} f_{R_1}(R_1, t) < \frac{\partial}{\partial x} f_{R_2}(R_1, t)$, thus by the maximum principle, $f_{R_1}(R_1,t) < f_{R_2}(R_1,t)$ for all time $t > t_1$. Hence there will be at most one extra intersection point.

For time $t \leq  t_1$, we must have $f_{R_1}(R_1,t) \geq  f_{R_2}(R_1,t)$, and argue similarly as above we get $u_{R_1}(0,t) < u_{R_2}(0,t)$.

For time $t > t_1$, we have $f_{R_1}(R_1,t) < f_{R_2}(R_1,t)$, and $\Gamma_{R_1,t}$ and $\Gamma'_{R_2,t}$ intersect at most at one point. If there is exactly one intersection point, then $u_{R_1}(0,t) < u_{R_2}(0,t)$. If there is no intersection point, $\Gamma'_{R_2,t}$ is on top of $\Gamma_{R_1,t}$, and by a similar argument as in Lemma \ref{neck point comparison}, we get a contradiction.

In summary, $u_{R_1}(0,t) \leq  u_{R_2}(0,t)$ for all $t\geq 0$.

\smallskip

\emph{Case} $2$. If $u_{R_1}(0,0) = u_{R_2}(0,0) = a$. For any small time $\e > 0$, since the height is strictly decreasing, $f_{R_1}(R_1, \e) < f_{R_2}(R_1, 0)$.

If $u_{R_1}(0, \e) \geq u_{R_2}(0,0)$, then $\Gamma'_{R_2,0}$ is on top of $\Gamma_{R_1,t}$. By a similar argument as in Lemma \ref{neck point comparison}, we get a contradiction. Hence $u_{R_1}(0, \e) < u_{R_2}(0,0)$, which implies that $\Gamma_{R_1, \e}$ and $\Gamma'_{R_2, 0}$ have exactly one intersection point.

Since $f_{R_1}(R_1, \e) < f_{R_2}(R_1, 0)$ and $\frac{\partial}{\partial x} f_{R_1}(R_1, t+ \e) = 0 < \frac{\partial}{\partial x} f_{R_2}(R_1, t)$, by Sturmian Theorem, $\Gamma_{R_1, t + \e}$ and $\Gamma'_{R_2, t}$ have at most one intersection point. By maximum principle, $f_{R_1}(R_1, t + \e) < f_{R_2}(R_1, t)$, therefore if there is exactly one intersection point, we get $u_{R_1}(0, t+ \e) < u_{R_2}(0,t)$. If there is no intersection point, $\Gamma'_{R_2, t}$ is on top of $\Gamma_{R_1, t + \e}$. Again, by a similar argument as in Lemma \ref{neck point comparison}, we get a contradiction.

Therefore, $u_{R_1}(0, t+ \e) < u_{R_2}(0,t)$ is true for all $\e > 0$, by taking the limit as $\e \to 0$, we claim that $u_{R_1}(0, t) \leq u_{R_2}(0,t)$.

\end{proof}

From Remark \ref{neck point upper bound}, we know that $u_R(0, t) \leq \alpha$ for all $R > 2\alpha$. Thus by Proposition \ref{prop：neck comparison for different R}, we know that $\eta(t) = \lim\limits_{R \to \infty} u_R(0,t)$ exists.

Let $\tau > 2\alpha +2$ be a large constant, then $u_{\tau}(0,t) \leq \eta(t) \leq \alpha$ for all $t \geq 0$. Given any $T > 0$, let $\mu(T) = \min_{t \in [0,T]} u_{\tau} (0,t) > 0$. Then for any $R > \tau$, $t \in [0,T]$, we have $\mu(T) \leq u_R(0,t) \leq \alpha$.

\begin{prop} \label{neck uniform gradient estimate}

Given $R > \tau$, $T > 0$. For any $y \in [0, \frac{\mu(T)}{2}]$, $t \in [0,T]$, we have
\begin{equation} \label{vertical uniform gradient bound}
\frac{\pr}{\pr y} u_R(y, t) \leq \frac{2\alpha e^{\frac{2y}{\mu(T)}}}{\mu(T)} \leq \frac{2 \alpha e}{\mu(T)}.
\end{equation}
\end{prop}

\begin{proof}

When $t = 0$, $\frac{\pr}{\pr y} u_R(y, t) = 0 \leq \frac{2\alpha e^{\frac{2y}{\mu(T)}}}{\mu(T)}$. Now we focus on $t>0$.

Let $h_\xi(y) = \frac{\mu(T)}{2} \cosh(\frac{2(y - \xi)}{\mu(T)})$, defined on $[\xi, \infty)$. Its inverse function is given by $h_\xi^{-1}(x) = \xi + \frac{\mu(T)}{2} \ln (\frac{2x}{\mu(T)} + \sqrt{\frac{4x^2 - \mu(T)^2}{\mu(T)^2}})$.

Since $u_R(0,t) \geq \mu(T)$ for all $t \in [0,T]$, by Corollary \ref{lower bound on height}, $u_R(y,t)$ is well-defined for $y \in [0,\frac{\mu(T)}{2}]$, $t \in (0, T]$.

Consider the catenoid with neck point $(\frac{\mu(T)}{2}, \xi)$. It is static under the mean curvature flow, and its profile curve $\overline{\mathcal{C}}_{\xi}$ is given by $x = h_\xi(y)$. Let $\mathcal{C}_{\xi}$ denote the graph of $y = h_\xi^{-1}(x)$, and $\overline{\mathcal{C}}_{\xi}$ is the union of $\mathcal{C}_{\xi}$ and its reflection with respect to the line $y = \xi$.

By the comparison principle Proposition \ref{comparison}, for any $t \geq 0$, $\mathcal{C}_{\xi}$ can not be on top of $\Gamma_{R,t}$. Therefore there exists $x(t) \in [u_R(0,t) ,R ]$ such that $h_\xi^{-1}(x(t)) < f_R(x(t),t)$.

Let $c = h_0^{-1}(R)$. The $y-$coordinate of the intersection point between $\mathcal{C}_{\xi}$ and the boundary of the cylinder $\{x = R\}$ is $(R,\xi + c)$. 

$\Gamma_{R,0} = \rho_{n_R,R}$ with $n_R \geq \mu(T)$. From the construction of $\rho_{\delta, R}$ in \eqref{initial curve definition}, we observe that for $\xi \geq f_R(R,0) - c$, there are at most two intersection points between $\Gamma_{R,0}$ and $\mathcal{C}_{\xi}$; for $\xi <  f_R(R,0) - c$, there is at most one intersection point between $\Gamma_{R,0}$ and $\mathcal{C}_{\xi}$. 

We have the following claims:

\smallskip

\noindent\textbf{Claim 3.} For any $t \in [0,T]$, there are at most two intersection points between $\Gamma_{R,t}$ and $\mathcal{C}_{\xi}$.

\noindent \emph{Proof.} By the Sturmian Theorem, the new intersection point between $\Gamma_{R,t}$ and $\mathcal{C}_{\xi}$ can only appear on the boundary $x = R$. As mentioned above, the boundary point of $\mathcal{C}_{\xi}$ is $(R, \xi + c)$.

If $\xi \geq f_R(R,0) - c$, then for any time $t > 0$, $f_R(R,t) < f_R(R,0) \leq \xi + c$, there will be no extra intersection point appearing on the boundary.

If $\xi <  f_R(R,0) - c$, by Lemma \ref{height decreasing}, the boundary point of $\Gamma_{R,t}$ moves toward $(R,0)$ monotonically under the flow, thus there will be at most one extra intersection point.

This proves the claim. \hfill\qed

\smallskip

Given time $t_1 \in (0,T]$, let $D(y) = h_0^{-1}(u_R(y,t_1)) - y$, which is well defined on $[0, \frac{\mu(T)}{2}]$. Due to the uniqueness and smoothness of the solution $f_R(x,t)$, along with the reflexive symmetry of the profile curve, we have $\frac{\partial}{\partial y} u_R(0,t_1) = 0$. Then 
\begin{align*}
D'(0) = (h_0^{-1})'(u_R(0,t_1)) \frac{\pr}{\pr y} u_R(0,t_1) - 1 = -1 < 0.
\end{align*}
Since $D'(x)$ is a smooth function, then there exists a maximal interval $[0, q]$ such that $D'(y)$ is nonpositive on this interval. Then $D(0) - D(q) = - \int_0^q D'(y) dy > 0$.

\noindent\textbf{Claim 4.} The upper bound of the maximal interval $q \geq \frac{\mu(T)}{2}$.

\noindent \emph{Proof.} We argue by contradiction. If $q < \frac{\mu(T)}{2}$, by the choice of $q$, there exists a small $\e_1 > 0$ with $q + \e_1 < \frac{\mu(T)}{2}$ such that $D'(q + \e_1) > 0$, and $|D(q) - D(q + \e_1)| < \frac{d}{2}$.

Thus $D(0) > D(q + \e_1)$. By the inverse function theorem, 
\begin{align*}
(h_0^{-1})'(u_R(q +\e_1,t_1)) > \left(\frac{\pr}{\pr y} u_R(q+\e_1, t_1)\right)^{-1} > \frac{\pr}{\pr x} f_R(u_R(q+\e_1, t_1), t_1).
\end{align*}

let $\xi = - D(q + \e_1)$, then $\mathcal{\mathcal{C}}_{\xi}$ and $\Gamma_{R,t_1}$ intersect at $(u_R(q + \e_1, t_1), q + \e_1)$.

We know $\xi = q+ \e_1 - h_0^{-1}(u_R(q + \e_1), t_1)) \leq q + \e_1 - h_0^{-1}(\mu(T)) < \frac{\mu(T)}{2} - \frac{3}{5} \mu(T) < 0$.

By the comparison principle Proposition \ref{comparison}, $\mathcal{C}_0$ can not be on top of $\Gamma_{R,t_1}$. Thus there exists $x_1 \in [u_R(0,t) ,R ]$ such that $h_0^{-1}(x_1) < f_R(x_1,t_1)$. Then
\begin{equation*}
\begin{aligned}
x_1 &= u_R(f_R(x_1, t_1), t_1) > u_R(h_0^{-1}(x_1),t_1) \geq u_R(h_0^{-1} (u_R(0,t)), t_1) \\
& \geq  u_R(h_0^{-1}(\mu(T)), t_1) > u_R(\frac{3}{5}\mu(T), t_1).
\end{aligned}
\end{equation*}

Since $\frac{\partial}{\partial x} f_R(u_R(q+\e_1, t_1),t_1) < (h_0^{-1})'(u_R(q +\e_1,t_1))$, then there exists a small $\e_2 > 0$ with $\e_2 < \min \{ u_R(q+\e_1, t_1) - u_R(0, t_1) , u_R(\frac{3}{5}\mu(T), t_1) - u_R(q + \e_1, t_1) \}$ such that
\begin{equation*}
\begin{aligned}
& f_R(u_R(q+\e_1, t_1) - \e_2,t_1) > h_i^{-1}(u_R(q +\e_1,t_1) - \e_2), \\
& f_R(u_R(q+\e_1, t_1) + \e_2,t_1) < h_i^{-1}(u_R(q +\e_1,t_1) + \e_2).
\end{aligned}
\end{equation*}

Now for $x = u_R(0, t_1)$ and $x = x_1$, we have
\begin{equation*}
\begin{aligned}
&f_R(u_R(0,t_1),t_1) = - D(0) + h_0^{-1}(u_R(0,t_1)) < -D(q+\e_1) + h_0^{-1}((u_R(0,t_1))) = h_i^{-1} (u_R(0,t_1)), \\
&f_R(x_1, t_1) > h_0^{-1} (x_1) > h_i^{-1}(x_1).
\end{aligned}
\end{equation*}

By the intermediate value theorem, there exists $x_2 \in (u_R(0,t_1), u_R(q+\e_1, t_1) - \e_2), x_3 \in (u_R(q+\e_1, t_1) + \e_2, x_1)$ such that $f_R(x_2, t_1) = h_0^{-1}(x_2), f_R(x_3, t_1) = h_0^{-1}(x_3)$. Therefore there are three intersection points $(u_R(q+\e_1, t_1), q + \e_1), (x_2, f_R(x_2, t_1)), (x_3, f_R(x_3,t_1)$ between $\Gamma_{R,t_1}$ and $\mathcal{C}_{\xi}$, which contradicts to Claim $3$. \hfill\qed

\smallskip

Now for any $y \in [0, \frac{\mu(T)}{2}]$ and $t \in (0,T]$, by Claim $4$, $D'(y) \leq 0$. Thus
\begin{equation*}
\frac{\pr}{\pr y} u_R(y, t) \leq \frac{1}{(h_0^{-1})'(u_R(y,t))} = \frac{\sqrt{4 (u_R(y,t))^2 - \mu(T)^2}}{\mu(T)} \leq \frac{2 u_R(y,t)}{\mu(T)}.
\end{equation*}

Since $u_R(0,t) \leq \alpha$, thus $u_R(y,t) \leq \alpha e^{\frac{2y}{\mu(T)}}$, and we prove \eqref{vertical uniform gradient bound}.

\end{proof}

Now we are ready to take the limit of $f_R(x,t)$ as $R \to \infty$.

\begin{prop}
There exists a rotationally symmetric mean curvature flow that is the $C_{\text{loc}}^\infty$ subsequential limit of $S(G_{f_R(\cdot,t)})$ as $R\to\infty$. This limit flow exists for all future time.
\end{prop}

\begin{proof}

Given $T > 0$, for any $R > \tau$, $t_1 \in (0,T]$, the neck point of $f_R(x,t_1)$ is $(u_R(0,t_1), 0)$. Because $f_R(x,t_1)$ is a strictly increasing function, for any $x_1 \geq u_{R}(0,t_1)$, the restricted curve $\{(x, f_R(x,t_1))| x \geq x_1 \}$ lies within the region $\{x \geq x_1, y \geq f_R(x_1,t_1)\}$. Hence the intersection angle between this restricted curve and the line $\{y = x + f_R(x_1, t_1) - x_1\}$ is not larger than $\frac{\pi}{4}$. 

Consider the graph of function $f_R(x,t_1)$ on the interval $[u_{R}(0,t_1), \alpha  + 4]$. We can view it as the graph of a function $g(x)$ over the line $\{y = x\}$. The above argument tells us that $|g'(x)| \leq 1$. The two endpoints of this restricted graph are $(u_R(0,t_1), 0)$ and $( \alpha + 4, f_R(\alpha + 4, t_1))$. The $x-$coordinates of the projection of these two points onto $\{y = x\}$ have the following bounds:
\begin{equation*}
\begin{aligned}
0 < &\frac{u_R(0,t_1)}{2} \leq \frac{\alpha}{2}, \qquad
\frac{\alpha + 4}{2} < &\frac{\alpha + 4 + f_R(\alpha + 4, t_1)}{2} < \frac{\alpha + 4 + \alpha}{2}.
\end{aligned}
\end{equation*}

Thus this restricted graph can be viewed as a normal graph of a $C^1$ function with a uniform derivative bound (not dependent on $R$ and $t$). Together with the uniform gradient estimate given by Proposition \ref{large x uniform gradient estiamte}, for any compact region $Q: \{(x,t) : 0 \leq x \leq a, 0 \leq t \leq b\}$, we can take a subsequential limit of the restricted function $f_R(x,t)|_{Q}$ as $R \to \infty$ by Arzela-Ascoli theorem. By a diagonal sequence argument, we know $\{f_R(x,t), t \in [0,T
]\}$ subsequentially converges in $R$ to a function $\tilde{f}(x,t)$. By the uniform gradient estimates in Proposition \ref{large x uniform gradient estiamte} and Proposition \ref{neck uniform gradient estimate}, the interior gradient estimate in \cite{EckerHuisken91_interior}, this convergence is $C_\text{loc}^\infty$ in the spacetime.

From the construction, we can see that $\tilde{f}$ is defined for all $t \in [0,T]$, $x \in [\eta(t), \infty)$. The fact that $\eta(t) \geq u_{\tau}(0,t) > 0$ with a further diagonal argument guarantees a limit mean curvature flow that exists for all future time.

\end{proof}

\begin{rmk}
Since $\tf$ is a subsequential limit of $f_R$, thus all the uniform properties of $f_R$ also hold for $\tf$. Hence $\tf$ is a strictly increasing function in $x$ for all $t > 0$, its function value is bounded by $\alpha$, and its neck point tends to the origin monotonically. The gradient estimate in Proposition \ref{large x uniform gradient estiamte} also holds for $\tf$.
\end{rmk}

So far, we have obtained an eternal flow $M(t):=S(G_{\tf(\cdot,t)})$ defined in the whole $\R^3$. It remains to show that the long time limit of this flow is the plane with multiplicity $2$.

We first show that for a sequence of $t_i\to\infty$, $M(t_i)$ converges to the plane with multiplicity $2$. As a consequence, this allows us to show that the neck point of $\tf(\cdot,t)$ tends to $0$ as $t\to\infty$. 

The proof uses an argument by Ilmanen \cite{Ilmanen95_Sing2D}, saying that if the spacetime $L^2$-norm of the mean curvature of a sequence of embedded mean curvature flows in $\R^3$ tends to $0$, then this sequence of mean curvature flows subsequentially converge to a smooth embedded minimal surface with multiplicity. The following statement can be found in \cite[Lemma 2.13]{BamlerKleiner23_multiplicity}, with slight modifications in notations.

\begin{lem}\label{lem:Ilmanen argument}
    Suppose $\delta>0$, $(M^i(t))_{t\in[t_1-\delta,t_2+\delta]}$ is a sequence of smooth mean curvature flows in $\R^3$, with the genus of $M^i(t)$ and the Gaussian density of $M^i(t)$ uniformly bounded for all $i\in\Z_+$ and $t\in[t_1-\delta,t_2+\delta]$. Suppose for every open subset $U$ of $\R^3$, we have
    \begin{equation}\label{eq:H L^2 bound}
        \int_{t_1}^{t_2}\int_{M^i(t)\cap U}|\vec H|^2d\mathcal H^2dt\to 0,\ \text{as $i\to\infty$.}
    \end{equation}
    Then there exists a subsequence of $(M^i(t))_{t\in[t_1,t_2]}$ that converges (in the sense of varifold convergence) to a static mean curvature flow, which is a smooth embedded complete minimal surface with multiplicity $k\in\Z_+$.
\end{lem}

Given $R>0$. To obtain the local $L^2$-bound \eqref{eq:H L^2 bound} of the limit flow $M(t)$ in $C_R$, we need to know the behavior of the mean curvature of $M(t)$ on the boundary of $C_R$.

\begin{lem} \label{non positive mean curvature}

For any $x \geq \alpha + 2$, $t > 0$, the mean curvature vector $\vec H$ of the upper half of $S(G_{\tf(\cdot,t)})$ points downward at the point $\tilde{f} (x,t)$. More precisely, the projection of $\vec{H}$ to the $y-$axis is negative.

\end{lem}

\begin{proof}

Consider the family of the translations of the catenoids given by the graph of the function $U_{\nu, \xi} (x) = \xi + \nu \ln (\frac{x}{\nu} + \sqrt{\frac{x^2 - \nu^2}{\nu^2}})$, where $0 < \nu \leq \alpha + 1$.

Given $x_1 \geq \alpha + 2$, $t_1 > 0$. By Proposition \ref{large x uniform gradient estiamte}, we know $0 < \frac{\pr}{\pr x} \tilde{f} (x_1, t_1) \leq \frac{\alpha +1}{\sqrt{x^2 - (\alpha + 1)^2}}$. Thus there exist $\nu$ and $\xi$ such that
\begin{align*}
U_{\nu,\xi}(x_1) = \tf(x_1, t_1),\ U'_{\nu,\xi}(x_1) = \frac{\pr}{\pr x} \tf(x_1, t_1).
\end{align*}

If $\frac{\pr^2}{\pr x^2} \tf(x_1, t_1) > U''_{\nu, \xi} (x_1)$, then the graph of $\tf(\cdot, t_1)$ is on top of the graph of $U_{\nu, \xi}$ near a neighborhood of $x_1$. Since $U_{\nu, \xi}$ tends to $\infty$ as $x \to \infty$ and $\tf(\cdot, t_1)$ is bounded, these two graphs have another intersection point far away from the origin. The mean curvature at $\tf(x_1, t_1)$ along its graph is strictly positive by the local comparison with the catenoid, thus at time $t_1 - \e$, for some $\e > 0$ small enough, the graph of $\tf(\cdot, t_1 - \e)$ and the graph of $U_{\nu, \xi}$ have two intersection points in a neighborhood of $\tf(x_1, t_1 - \e)$, while there is another intersection point far away from these two intersection points by the boundedness of $\tf$ and unboundedness of the catenoid.

But at $t = 0$, the graph of $\tf(\cdot, 0)$ and the graph of $U_{\nu, \xi}$ have at most two intersection points, and for large $x$, the second graph is always strictly on top of the first graph by the boundedness. By Sturmian Theorem, there are at most two intersection points between the graph of $\tf(\cdot, t_1 - \e)$ and the graph of $U_{\nu, \xi}$. This yields a contradiction.

Hence $\frac{\pr^2}{\pr x^2} \tf(x_1, t_1) \leq U''_{\nu, \xi} (x_1)$, which implies that locally $\tf(\cdot,t_1)$ lies below $U_{\nu,\xi}$. Therefore, locally the upper half of $S(G_{\tf(\cdot,t_1)})$ touches the catenoid from below. This shows that the mean curvature vector of the upper half of $S(G_{\tf(\cdot,t_1)})$ points downward.
\end{proof}

\begin{prop} \label{limit neck point limit}
For any sequence $t_i\nearrow\infty$, there exists a subsequence $t_{i_j}$ that $M(t_{i_j})$ converges (in the sense of varifold convergence) to the union of planes with possibly higher multiplicity. Moreover, the neck point of $\tf$ tends to the origin as $t \to \infty$.
\end{prop}

\begin{proof}
Without loss of generality, we remove some of $t_i$'s to assume $t_{i+1}-t_i>1$. Let $M^i(t)=(M(t+t_i))_{t\in[0,1]}$. Consider the restriction of our family of surfaces $M(t)$ on $C_R$, denoted by $\overline{M}(t)$. We use $\overline{M}^i(t)$ to denote the flow $(\overline{M}(t+t_i))_{t\in[0,1]}$. By the first variation formula, we have
\begin{align*}
\int_0^{T} \int_{\overline{M}(t)} |\vec H|^2 d\mathcal H^2 dt = \text{Area}(\overline{M}(0)) - \text{Area}(\overline{M}(T)) + \int_0^{\infty} \int_{\partial \overline{M}(t)} \langle \eta, \vec H \rangle ds dt,
\end{align*}
where $\eta$ is the outward unit co-normal of $\partial \overline{M}(t)$. By Lemma \ref{non positive mean curvature}, and $\frac{\pr}{\pr x} \tf(R,t) > 0$, we know $\langle \eta, \vec H \rangle \leq 0$ along $\partial \overline{M}^t$. Since $\overline{M}(t)$  is a rotationally symmetric surface generated by graphs of bounded increasing function on a bounded interval, thus the area of $\overline{M}(t)$ is bounded and independent of $t$. Hence $\int_0^{\infty} \int_{\overline{M}(t)} |\vec H|^2 d\mathcal H^2 dt$ is bounded, and moreover, $\int_0^1 \int_{\overline{M}^i(t)} |\vec H|^2 d\mathcal H^2 dt\to 0$ as $i\to\infty$. By the compactness of weak mean curvature flows (see Section 7 of \cite{Ilmanen94_elliptic}) and Lemma \ref{lem:Ilmanen argument}, we know that $M^i$ subsequentially converges to a static mean curvature flow defined for time $[0,1]$, which is a smooth embedded minimal surface $\Sigma$ with possibly higher multiplicity. Such $\Sigma$ is also rotationally symmetric and lies between two parallel planes as $M(t)$ does. The only smooth embedded minimal surface that satisfies all these properties is the plane or the union of parallel planes.

As a consequence, the neck point of $\tilde f$ tends to the origin as $t\to\infty$.
\end{proof}

\begin{remark}
    Ilmanen's argument only assures the regularity of the limit minimal surface. It does not ensure the convergence is also regular, for example, it does not ensure the convergence is in the Lipschitz topology.

    One reason is that there could be ``pimples'' in the sequence of flows that vanish in the geometric measure theory limit. Such pimples are described by Simon in \cite[Lemma 2.1]{Simon93_willmore}, which is a key lemma that is used by Ilmanen in \cite{Ilmanen95_Sing2D}. On the other hand, if we can show that any subsequential limit in Proposition \ref{limit neck point limit} is a fixed plane, we can use Brakke's regularity of mean curvature flow to show that the convergence is multiplicity $2$ in our sense. Nevertheless, we decided to show a quantitative gradient estimate directly, because such an estimate gives us a better understanding of this limit eternal flow.
\end{remark}

From the gradient estimate \eqref{horizontal uniform gradient bound}, we know that
\begin{equation} \label{limit flow gradient estimate}
\frac{\pr}{\pr x} \tf(x,t) \leq \frac{\alpha + 1}{\sqrt{x^2 - (\alpha + 1)^2}}, 
\end{equation}
for all $x \geq \alpha + 2$, $t \geq 0$. Next, we adapt the same method that we used in Proposition \ref{gradient estimate} to get an improved gradient estimate over any closed interval in $(a, \infty)$.

\begin{prop} \label{limit gradient estimate}
For any $0 < a < b < c$, choose $k > c + \alpha + 2$ large enough such that
\begin{equation} \label{k condition}
\frac{(\alpha + 1) (\ln k - \ln a)}{\sqrt{k^2 - (\alpha + 1)^2}} < \frac{\pi}{4} \ln \frac{b}{a}.
\end{equation}
Then there exist a constant $\omega$ depending on $a, k$ and a constant $T$ depending on $a, b, k$, such that
\begin{align*}
\frac{\partial}{\partial x} \tf (x,t) \leq \tan \left(\frac{\pi}{2} \frac{\ln k - \ln b}{\ln k -\ln a} + \frac{\omega - b}{\ln t} + \arctan\left(\frac{\alpha + 1}{\sqrt{k^2 - (\alpha + 1)^2}}\right)\right),
\end{align*}
for all $x \in [b,k]$, $t \geq T$. In addition, $\lim\limits_{t \to \infty} [\tf (c,t) - \tf (b,t)] = 0$.
\end{prop}

\begin{proof}

By Proposition \ref{limit neck point limit}, we know that there exists $T > 0$ such that for all $t \geq T$, the distance from the neck point of $\tf(x, t)$ to the origin is less than $\frac{a}{2}$. Hence $\tf(x,t)$ is a smooth function over $[a,k]$ for all $t \geq T$.

We know
\begin{align*}
\tf_t = \frac{\tf_{xx}}{1+\tf_x^2} + \frac{\tf_x}{x}, \qquad \tf_x > 0 \text{ for all } x \in [a,k].
\end{align*}

Let $\phi(x,t) = \arctan (\tf_x(x,t))$ for $x \in [a,k]$, $t \in [T, \infty)$. Then $0 \leq \phi(x,t) < \frac{\pi}{2}$, and $\tf_x(x,t) = \tan (\phi(x,t))$. As shown in the proof of Proposition \ref{gradient estimate}, 

\begin{equation}
\phi_t - \frac{\phi_x}{x} - \frac{1}{1+ \tf_x^2} \phi_{xx} = - \frac{\tf_x}{x^2(1+ \tf_x^2)} < 0.
\end{equation}

Let $\mu = \frac{\pi}{2 \ln \frac{k}{a}}$, $T' = T + k(k + \frac{\pi}{2})$, $\omega = k + \frac{\pi}{2} \ln T'$. Let
\begin{align*}
\varphi(x,t) = \mu \ln \frac{k}{x} + \frac{\omega - x}{\ln t} + \arctan \left(\frac{\alpha + 1}{\sqrt{k^2 - (\alpha + 1)^2}}\right).
\end{align*}

By the same argument as in the proof of Proposition \ref{gradient estimate}, we have $\varphi_t - \frac{\varphi_x}{x} - \frac{\varphi_{xx}}{1+\tf_x^2} > 0$.

Thus for $x \in [a,k]$, $t \geq T'$, we know that
\begin{align*}
\begin{aligned}
& \varphi(x,T') \geq \frac{\omega - x}{\ln T'} \geq \frac{\pi}{2} \geq \phi(x,T'), \\
&\varphi_t - \frac{\varphi_x}{x} - \frac{\varphi_{xx}}{1+\tf_x^2} > 0 > \phi_t - \frac{\phi_x}{x} - \frac{\phi_{xx}}{1+\tf_x^2} , \\
&\varphi(a,t) \geq \mu \ln \frac{k}{a} = \frac{\pi}{2} \geq \phi(a,t), \\
&\varphi(k,t) > \arctan\left(\frac{\alpha + 1}{\sqrt{k^2 - (\alpha + 1)^2}}\right) \geq \phi(k,t), \qquad (\text{by } \eqref{limit flow gradient estimate}).
\end{aligned}
\end{align*}

We can apply the classical maximum principle to conclude that $\varphi (x,t) \geq \phi(x,t)$ for all $x \in [a,R]$, $t \geq T'$.

For any $x \in [b, k]$, $\phi(x,t) \leq \varphi(x,t) \leq \varphi(b,t) = \frac{\pi}{2} \frac{\ln k - \ln b}{\ln k -\ln a} + \frac{\omega - b}{\ln t} + \arctan(\frac{\alpha + 1}{\sqrt{k^2 - (\alpha + 1)^2}})$. 

The condition \eqref{k condition} guarantees that $\arctan(\frac{\alpha + 1}{\sqrt{k^2 - (\alpha + 1)^2}}) < \frac{\pi}{4} (1- \frac{\ln k - \ln b}{\ln k -\ln a})$. In addition, there exists $T'' > T'$ such that for all $t \geq T''$, we have $\frac{\omega - b}{\ln t} < \frac{\pi}{4} (1- \frac{\ln k - \ln b}{\ln k -\ln a})$. Therefore for all $x \in [b, k]$, $\varphi(x,t) < \frac{\pi}{2}$.

For all $x \in [b,k]$, $t \geq T''$, we have
\begin{equation*}
\begin{aligned}
\tf_x (x,t) &= \tan (\phi(x,t)) \leq \tan \left(\frac{\pi}{2} \frac{\ln k - \ln b}{\ln k -\ln a} + \frac{\omega - b}{\ln t} + \arctan\left(\frac{\alpha + 1}{\sqrt{k^2 - (\alpha + 1)^2}}\right)\right), \\
\tf (c,t) - \tf (b,t) &= \int_b^c \tf_x(x,t) dx \\
&\leq (c - b) \tan \left(\frac{\pi}{2} \frac{\ln k - \ln b}{\ln k -\ln a} + \frac{\omega - b}{\ln t} + \arctan\left(\frac{\alpha + 1}{\sqrt{k^2 - (\alpha + 1)^2}}\right)\right).
\end{aligned}
\end{equation*}

Let $t \to \infty$, we know
\begin{align*}
0 \leq \limsup\limits_{t \to \infty} [\tf (c,t) - \tf (b,t)] \leq (c - b) \tan \left(\frac{\pi}{2} \frac{\ln k - \ln b}{\ln k -\ln a} + \arctan\left(\frac{\alpha + 1}{\sqrt{k^2 - (\alpha + 1)^2}}\right)\right).
\end{align*}

Since the condition \eqref{k condition} works for all $k$ large enough, $a$ small enough, thus we can let $a \to 0$, and $k \to \infty$. Let $a \to 0$, we have $0 \leq \limsup\limits_{t \to \infty} [\tf (c,t) - \tf (b,t)] \leq \frac{(c - b)(\alpha + 1)}{\sqrt{k^2 - (\alpha + 1)^2}}$. Then let $k \to \infty$, we conclude 
$
\lim\limits_{t \to \infty} [\tf(c,t) - \tf(b,t)] = 0$.
\end{proof}

We are now prepared to prove our main theorem.

\begin{thm} \label{main}
There exists a rotationally symmetric surface in $\R^3$, the mean curvature flow starting from this surface exists for all future time, and the forward limit is the multiplicity $2$ plane $\{y = 0\}$.    
\end{thm}

\begin{proof}

We have already shown the mean curvature flow starting from $\rho_{\eta}$, given by $S(G_{\tf(x,t)})$,  exists for all future time. Now we show that it converges to the multiplicity $2$ plane as $t \to \infty$.

For any subsequential limit of $\tf(x,t)$, by Proposition \ref{limit neck point limit}, we know it is the graph of a function defined on $[0, \infty)$. Lemma \ref{height upper bound} implies that $\tf (x,t) \leq \alpha x$, for all $t \in [0, \infty)$. By proposition \ref{limit gradient estimate}, we know that for any $0 < b < c$, we have
\begin{align*}
0 \leq \limsup\limits_{t \to \infty} \tf(c,t) = \limsup\limits_{t \to \infty}  \tf(b,t) \leq \alpha b   
\end{align*}

Since $b$ can be arbitrarily small, hence for all $x > 0$, any subsequential limit of $\tf(x,t)$ as $t \to \infty$ is $0$. Moreover, the proof of Proposition \ref{limit flow gradient estimate} also implies the gradient $\frac{\partial}{\partial x}\tf(x,t)$ uniformly converges to $0$ at $t\to\infty$ for $x$ in a fixed compact sub-interval of $(0,\infty)$. Therefore, the forward limit of the immortal flow is the multiplicity $2$ plane $\{y = 0\}$.
\end{proof}

\begin{rmk}\label{rmk:different family}
If we replace the family of curves ${\rho_{\delta}}$ with a family of quarter circles, namely $\{(R - x, \sqrt{r^2 - x^2}) |\ x \in [0, r]\}$ for $0 \leq r \leq R$, then the interpolation argument as described in Theorem \ref{free boundary immortal flow} (with a possibly different bound on $R$) remains valid. Therefore, for large $R$, we can obtain a family of rotationally symmetric free boundary surfaces in $C_R$, evolving under the mean curvature flow, and converging to a multiplicity $2$ disk.

However, if we consider the surface in this family with the neck point at $(1,0)$ for large $R$, and take the limit as $R$ approaches infinity, these surfaces will converge to the catenoid that is formed by rotating the graph of $x = \cosh y$. This addresses the importance of making a careful choice of the (singular) foliation.
\end{rmk}

\appendix
\section{Asymptotically planar mean curvature flow}

\begin{definition}
    A surface $M\subset\R^3$ is asymptotic to a plane $P$ if for any $\epsilon>0$ there exists a compact set $K_\epsilon$ such that $\overline{M\backslash K_\epsilon}$ is a complete noncompact surface with boundary, $\partial M\subset K_\epsilon$, and $\overline{M\backslash K_\epsilon}\subset P_\epsilon$, where $P_\epsilon$ is the $\epsilon$ tubular neighbourhood of $P$.
\end{definition}

\begin{prop} \label{asymptotic to the same plane}
    Suppose $\{M(t)\}_{t\in[0,1]}$ is a mean curvature flow, $R>0$, and $N(t)$ is a connected component of $M(t)\backslash \overline B_R$. If $N(0)$ is asymptotic to a plane $P$, then $N(t)$ is asymptotic to the same plane $P$ for $t\in[0,1]$.
\end{prop}

\begin{proof}
    Without loss of generality, we assume $P=\{(x,y,0)|(x,y)\in\R^2\}$. For any $\epsilon>0$, it suffices to find $R_\epsilon>0$, such that $\overline{N(t)\backslash B(R)}\subset P_{\epsilon}$. Because $N(0)$ is asymptotic to $P$, we may assume $N(0)\backslash B_{R'}\subset P_{\epsilon/4}$. Now we choose $r$ large such that $r-\sqrt{r^2-4}\leq \epsilon/2$ (e.g. $r^2>8/\epsilon+1$), and we choose $R_\epsilon>r+R'$. Then for any $(x,y)$ with $|(x,y)|\geq R_\epsilon$, $\partial B_r(x,y,\epsilon/2+r)$ has distance at least $\epsilon/2$ away from $P_{\epsilon/4}$, therefore it is disjoint from $N(0)\backslash B_{R'}$. Applying the avoidance principle shows that $N(t)\backslash B_{R_{\epsilon}}$ is disjoint from $\partial B_{\sqrt{r^2-4t}}(x,y,\epsilon/2+r)$. In particular, this implies that $N(t)\backslash B_{R_{\epsilon}}\subset P_{\epsilon}$ for $t\in[0,1]$.
\end{proof}

\bibliography{main}
\bibliographystyle{alpha}

\end{document}